\documentclass{amsart} 

\usepackage{multirow}
\usepackage{amssymb}
\usepackage{mathabx}
\usepackage{eucal} 
\usepackage{color}
\usepackage[bookmarks,colorlinks=true]{hyperref}

\newcommand \seq[1]{{\left\langle{#1}\right\rangle}}
 
\newcommand{\rest}[1]{\!\!\upharpoonright\!{#1}} 
\newcommand \tth{{}^{\textup{th}}}
\DeclareMathOperator \dom{dom}

\newcommand{\converge}{\!\!\downarrow}

\newcommand{\w}{\omega}
\newcommand \s{\sigma}
\renewcommand \le {\leqslant}
\renewcommand \ge {\geqslant}

\let\littlehat\hat
\renewcommand \hat{\widehat}

\renewcommand \bar{\overline}
\newcommand \conc[2]{{#1}\littlehat{\,\,}{#2}}

\newcommand \andd{\,\,\,\&\,\,\,}

\newcommand \R{\mathbb R}
\newcommand \Rat{\mathbb Q}

\newcommand \ZZ{\mathcal Z}
\newcommand \TT{\mathcal T}
\newcommand \II{\mathcal I}

\DeclareMathOperator \range{range}

\newcommand \DemBLR{\textup{Demuth}_{\textup{BLR}}}

\newcommand\ba{{\mathbf{a}}}

\theoremstyle{plain}
\newtheorem{theorem}{Theorem}[section] 
 
\newtheorem{proposition}[theorem]{Proposition} 
\newtheorem{lemma}[theorem]{Lemma} 
\newtheorem{corollary}[theorem]{Corollary} 
 
\newtheorem{claim}[theorem]{Claim}

\theoremstyle{definition}
\newtheorem{definition}[theorem]{Definition} 

\theoremstyle{remark}

\numberwithin{equation}{section}

\title{Inherent enumerability of strong jump-traceability}

\author{David Diamondstone}
\address{School of Mathematics, Statistics and Operations Research, Victoria University of Wellington,
  Wellington, New Zealand}
\email{\href{mailto:ddiamondstone@gmail.com}%
{ddiamondstone@gmail.com}}

\author{Noam Greenberg}
\address{School of Mathematics, Statistics and Operations Research, Victoria University of Wellington,
  Wellington, New Zealand}
\email{\href{mailto:greenberg@msor.vuw.ac.nz}%
{greenberg@msor.vuw.ac.nz}}
\urladdr{\url{http://homepages.mcs.vuw.ac.nz/~greenberg/}}

\author{Daniel Turetsky}
\address{School of Mathematics, Statistics and Operations Research, Victoria University of Wellington,
  Wellington, New Zealand}

\thanks{All authors were supported by the Marsden Fund of New Zealand, the first and the third as postdoctoral fellows. The second author was also supported by a Rutherford Discovery Fellowship.}

\begin{document}

\begin{abstract}
	We show that every strongly jump-traceable set obeys every benign cost function. Moreover, we show that every strongly jump-traceable set is computable from a computably enumerable strongly jump-traceable set. This allows us to generalise properties of c.e.\ strongly jump-traceable sets to all such sets. For example, the strongly jump-traceable sets induce an ideal in the Turing degrees; the strongly jump-traceable sets are precisely those that are computable from all superlow Martin-L\"{o}f random sets; the strongly jump-traceable sets are precisely those that are a base for $\DemBLR$-randomness; and strong jump-traceability is equivalent to strong superlowness. 
\end{abstract}

\today
\maketitle

\section{Introduction}



An insight arising from the study of algorithmic randomness is that anti-random-ness is a notion of computational weakness. While the major question driving the development of effective randomness was ``what does it mean for an infinite binary sequence to be random?'', fairly early on Solovay \cite{Solovay:manuscript} defined the notion of $K$-trivial sets, which are the opposite of Martin-L\"of random sequences in that the prefix-free Kolmogorov complexity of their initial segments is as low as possible. While Chaitin \cite{Chaitin_KTrivial,Chaitin:_information_strings} showed that each $K$-trivial set must be $\Delta^0_2$, a proper understanding of these sets has only come recently through work of Nies and his collaborators (see for example \cite{DHNS:Trivial,Nies:LownessRandomness,Nies:EliminatingConcepts,HirschfeldtNiesStephan:UsingRandomOracles}). This work has revealed that $K$-triviality is equivalent to a variety of other notions, such as lowness for Martin-L\"of randomness, lowness for $K$, and being a base for 1-randomness. These other notions express computational weakness, either as the target of a computation or as an oracle: they either say that a set is very easy to compute, or is a weak oracle and cannot compute much. 

The computational weakness of $K$-trivial sets is reflected in more traditional measures of weakness studied in pure computability theory. For example, every $K$-trivial set has a low Turing degree. Recent developments in both pure computability and in its application to the study of randomness have devised other notions of computational weakness, and even hierarchies of weakness, and attempted to calibrate $K$-triviality with these notions. One such attempt uses the hierarchy of \emph{jump-traceability}. 

While originating in set theory (see \cite{Raisonnier}), the study of traceability in computability was initiated by Terwijn and Zambella \cite{Terwijn:PhD,TerwijnZambella}.  

\begin{definition}\label{def:trace}
	A \emph{trace} for a partial function $\psi\colon \w\to \w$ is a sequence $T=\seq{T(z)}_{z<\w}$ of finite sets such that for all $z\in \dom \psi$, $\psi(z)\in T(z)$. 
\end{definition}

Thus, a trace for a partial function $\psi$ indirectly specifies the values of $\psi$ by providing finitely many possibilities for each value; it provides a way of ``guessing'' the values of the function $\psi$. Such a trace is useful if it is easier to compute than the function $\psi$ itself. 
In some sense the notion of a trace is quite old in computability 
theory. W.\ Miller and Martin  \cite{MillerMartin:Hyperimmune} characterised the hyperimmune-free degrees
as those Turing degrees ${\ba}$ such that every (total) function $h\in {\ba}$ has a computable trace (the more familiar, but equivalent, formulation, is in terms of domination). In the same spirit, Terwijn and Zambella used a uniform version of hyperimmunity to characterise lowness for Schnorr randomness, thereby giving a 
``combinatorial'' characterisation of this lowness notion.

In this paper we are concerned not with how hard it is to compute a trace, but rather, how hard it is to enumerate it. 

\begin{definition}\label{def: c.e. trace} 
	A trace $T = \seq{T(z)}$ is \emph{computably enumerable} if the set of pairs $\left\{ (x,z)  \,:\, x\in T(z) \right\}$ is c.e. 
\end{definition}
In other words, if uniformly in $z$, we can enumerate the elements of $T(z)$. It is guaranteed that each set $T(z)$ is finite, and yet if $T$ is merely c.e., we do not expect to know when the enumeration of $T(z)$ ends. Thus, rather than using the exact size of each element of the trace, we use effective bounds on this size to indicate how strong a trace is: the fewer options for the value of a function, the closer we are to knowing what that value is. The bounds are known as order functions; they calibrate rates of growth of computable functions.

\begin{definition}\label{def:orders}
	An \emph{order function} is a nondecreasing, computable and unbounded function $h$ such that $h(0)>0$. If $h$ is an order function and $T=\seq{T(z)}$ is a trace, then we say that $T$ is an \emph{$h$-trace} (or that $T$ is \emph{bounded by $h$}) if for all $z$, $|T(z)|\le h(z)$. 
\end{definition}

In addition to measuring the sizes of c.e.\ traces, order functions are used to define uniform versions of traceability notions. For example, \emph{computable traceability}, the uniform version of hyperimmunity used by Terwijn and Zambella, is defined by requiring that traces for functions in a hyperimmune degree $\ba$ are all bounded by a single order function.

Zambella (see Terwijn \cite{Terwijn:PhD}) observed that if $A$ is low for Martin-L\"of randomness then there is an order function $h$ such that every function computable from $A$ has a c.e.\ $h$-trace. This was improved by Nies \cite{Nies:LownessRandomness}, who showed that one can replace total by partial functions. In some sense it is natural to expect a connection between uniform traceability and $K$-triviality; if every function computable (or partial computable) from $A$ has a c.e.\ $h$-trace, for some slow-growing order function $h$, then the value $\psi(n)$ of any such function can be described by $\log n + \log h(n)$ many bits. 

Following this, it was a natural goal to characterise  $K$-triviality by tracing,  probably with respect to a family of order functions. While partial results have been obtained \cite{GeorgeRodNoam:KTriv_and_jump_trace,MerkleHoelzelKraeling:TimeBounded} this problem still remains open. The point is that while $K$-triviality has been found to have multiple equivalent definitions, all of these definitions use analytic notions such as Lebesgue measure or prefix-free Kolmogorov complexity in a fundamental way, and the aim is to find a purely combinatorial characterisation for this class.

An attempt toward a solution of this problem lead to the introduction of what seems now a fairly fundamental concept, which is not only interesting in its own right, but now has been shown to have deep connections with randomness.
\begin{definition}[Figuiera, Nies, and Stephan \cite{FigueiraNiesStephan:SJT}] \label{def:sjt}
	Let $h$ be an order function. An oracle $A\in 2^\w$ is \emph{$h$-jump-traceable} if every $A$-partial computable function has a c.e.\ $h$-trace. An oracle is \emph{strongly jump-traceable} if it is $h$-jump-traceable for every order function $h$. 
\end{definition}

Figueira, Nies, and Stephan gave a construction of a non-computable strongly jump-traceable c.e.\ set. Their construction bore a strong resemblance to the construction of a $K$-trivial c.e.\ set. J.\ Miller and Nies \cite{MillerNies:Open} asked if strong jump-traceability and $K$-triviality coincided.

Cholak, Downey, and Greenberg \cite{CholakDowneyGreenberg:SJT1} answered this question in the negative. They showed however that one implication holds, at least for c.e.\ sets: the strongly jump-traceable c.e.\ sets form a proper subclass of the c.e.\ $K$-trivial sets. They also showed that restricted to c.e.\ sets, the strongly jump-traceable sets share a pleasing feature with the $K$-trivials, in that they induce an ideal in the c.e.\ Turing degrees. 

In view of these results it might seem that strong jump-traceability might be an interesting artifact of the studies of randomness, but as it turned out, the class of c.e., strongly jump-traceable sets has been shown to have remarkable connections with randomness. Greenberg, Hirschfeldt, and Nies \cite{GreenbergHirschfeldtNies} proved that a c.e.\ set is strongly jump-traceable if and only if it is computable from every superlow random sets, if and only if it is computable from every superhigh random set. Greenberg and Turetsky \cite{GreenbergTuretsky:Demuth} complemented work of Ku{\v c}era and Nies \cite{KuceraNies:Demuth} and showed that a c.e.\ set is strongly jump-traceable if and only if it is computable from a Demuth random set, thus solving the Demuth analogue of the random covering problem, which remains open for Martin-L\"{o}f randomness and $K$-triviality.

The restriction to c.e.\ sets appeared to be a major technical drawback. The major tool introduced in \cite{CholakDowneyGreenberg:SJT1} for working with strongly jump-traceable oracles, called the \emph{box-promotion} method, works well for c.e.\ oracles; but technical difficulties restricted its application for other sets. Early on, Downey and Greenberg showed that all strongly jump-traceable sets are $\Delta^0_2$, and more recently in \cite{DGsjt2}, they showed that all such sets are in fact $K$-trivial, giving a full implication, not restricted to c.e.\ sets. In this paper we show how to overcome the difficulties in adapting the box-promotion method to work with arbitrary strongly jump-traceable oracles and to yield the following definitive result.

\begin{theorem}\label{thm_main_ce}
	Every strongly jump-traceable set is computable from a c.e., strongly jump-traceable set. 
\end{theorem}

This shows that strong jump-traceability, much like $K$-triviality, is \emph{inherently enumerable}. It cannot be obtained by devising a suitable notion of forcing, but essentially, only through a computable enumeration. While it is impossible for every strongly jump-traceable set to be c.e., as this notion is closed downward in the Turing degrees, Theorem \ref{thm_main_ce} says this downward closure is the only reason for the existence of non-c.e., strongly jump-traceable sets. 

\medskip

Theorem \ref{thm_main_ce} has a slew of corollaries. It enables us to extend characterisations of c.e.\ strong jump-traceability to all strongly jump-traceable sets. 

\begin{corollary}\label{cor_ideal}
	The Turing degrees of strongly jump-traceable sets form an ideal in the Turing degrees. 
\end{corollary}

\begin{proof}
	The Turing degrees of c.e., strongly jump-traceable sets form an ideal in the c.e.\ Turing degrees \cite{CholakDowneyGreenberg:SJT1}. 
\end{proof}

Figueira, Nies and Stephan introduced a notion seemingly stronger than strong jump-traceability, called \emph{strong superlowness}, which can be characterised using plain Kolmogorov complexity. 

\begin{corollary}\label{cor_SSL}
	A set is strongly jump-traceable if and only if it is strongly superlow. 
\end{corollary}

\begin{proof}
	Figueira, Nies, and Stephan \cite{FigueiraNiesStephan:SJT} showed that every strongly superlow set is strongly jump-traceable, and that the notions are equivalent on c.e.\ sets. Strong superlowness is also closed downward in the Turing degrees. 
\end{proof}

Unlike $K$-triviality, strong jump-traceability has both combinatorial and analytic characterisations. 

\begin{corollary}\label{cor_diamond}
	A set is strongly jump-traceable if and only if it is computable from all superlow Martin-L\"of random sets. 
\end{corollary}

\begin{proof}
	In \cite{GreenbergHirschfeldtNies} it is shown that every set computable from all superlow 1-random sets is strongly jump-traceable, and that every c.e., strongly jump-traceable set is computable from all superlow 1-random sets. 
\end{proof}

We remark that the results of \cite{GreenbergHirschfeldtNies} imply that every strongly jump-traceable set is computable from all superhigh random sets, but we do not yet know if all sets computable from all superhigh random sets are all strongly jump-traceable. 

Another connection between strong jump-traceability and randomness passes through a notion of randomness stronger than Martin-L\"of's, introduced by Demuth. As mentioned above, the Demuth analogue of the incomplete Martin-L\"of covering problem was solved by Greenberg and Turetsky, giving yet another characterisation of c.e.\ jump-traceability. This characterisation cannot, of course, be extended to all sets, since every Demuth random is computable from itself. The analogue of the covering problem for all sets is the notion of a \emph{base} for randomness: a set $A$ is a \emph{base} for a relativisable notion of randomness $\mathcal{R}$ if $A$ is computable from some $\mathcal{R}^A$-random set. Hirschfeldt, Nies and Stephan \cite{HirschfeldtNiesStephan:UsingRandomOracles} showed that a set is a base for Martin-L\"of randomness if and only if it is $K$-trivial. On the other hand, while every base for Demuth randomness is strongly jump-traceable (Nies \cite{Nies:Demuth}), these two notions do not coincide (Greenberg and Turetsky \cite{GreenbergTuretsky:Demuth}). However, this relies on the full relativisation of Demuth randomness. Recent work of Bienvenu, Downey, Greenberg, Nies and Turetsky \cite{BDGNT} discovered a partial relativisation of Demuth randomness, denoted $\DemBLR$, which is better behaved than its fully-relativised counterpart. 

\begin{corollary}\label{cor_Demuth}
	A set is strongly jump-traceable if and only if it is a base for $\DemBLR$-randomness. 
\end{corollary}

\begin{proof}
	Nies \cite{Nies:Demuth} showed that every set which is a base for Demuth randomness is strongly jump-traceable. An examination of his proof, though, shows that for the Demuth test he builds to use the hypothesis of being a base for Demuth randomness, the bounds he obtains are computable. In other words, his proof shows that every set which is a base for $\DemBLR$ randomness is strongly jump-traceable. 
	
	In the other direction, by \cite{GreenbergTuretsky:Demuth}, every c.e., strongly jump-traceable set $A$ is computable from a Demuth random set, and by \cite{BDGNT}, each such set is also low for $\DemBLR$ randomness, and so in fact computable from a $(\DemBLR)^A$-random set, in other words, is a base for $\DemBLR$ randomness. Again this notion is downwards closed in the Turing degrees. 
\end{proof}

\

Our proof of Theorem \ref{thm_main_ce} utilises a concept of independent interest, that of a \emph{cost function}. Formalised by Nies (see \cite{Nies:book}), cost function constructions generalise the familiar construction of a $K$-trivial set (see \cite{DHNT:Callibrating}) or the construction of a set low for $K$ (Mu\v{c}nik, see \cite{DowneyHirschfeldtBook}). Indeed, the key to the coincidence of $K$-triviality with lowness for $K$ is the fact that $K$-triviality can be characterised by obedience to a canonical cost function. 

In this paper, we define a \emph{cost function} to be a $\Delta^0_2$, non-increasing function from $\w$ to the non-negative real numbers $\R^+$. A cost function $c$ satisfies the \emph{limit condition} if its limit $\lim_x c(x)$ is 0. A \emph{monotone approximation} for a cost function $c$ is a uniformly computable sequence $\seq{c_s}$ of functions from $\w$ to the non-negative rational numbers $\Rat^+$ such that:
\begin{itemize}
	\item each function $c_s$ is non-increasing; and
	\item for each $x<\w$, the sequence $\seq{c_s(x)}_{s<\w}$ is non-decreasing and converges to $c(x)$.
\end{itemize}
Here we use the standard topology on $\R$ to define convergence, rather than the discrete topology which is usually used to define convergence of computable approximations of $\Delta^0_2$ sets and functions. A cost function is called \emph{monotone} if it has a monotone approximation. In this paper, \emph{we are only interested in monotone cost functions which satisfy the limit condition}, and so when we write ``cost function'', unless otherwise mentioned, we mean ``monotone cost function satisfying the limit condition''.

If $\seq{A_s}$ is a computable approximation of a $\Delta^0_2$ set $A$, then for each $s<\w$, we let $x_s$ be the least number $x$ such that $A_{s-1}(x)\ne A_s(x)$. If $\seq{c_s}$ is a monotone approximation for a cost function $c$, then we write $\sum c_s(A_s)$ for $\sum c_s(x_s)$. 
It is understood that if $A_s = A_{s-1}$, then no cost is added at stage $s$ to the sum $\sum c_s(A_s)$.

\begin{definition}\label{def_obedience}
	A $\Delta^0_2$ set $A$ \emph{obeys} a cost function $c$ if there is a computable approximation $\seq{A_s}$ of $A$ and a monotone approximation $\seq{c_s}$ of $c$ such that the sum $\sum_{s<\w} c_s(A_s)$ is finite. 
\end{definition}
Nies \cite{Nies:calculus} showed that obedience does not depend on the monotone approximation for $c$; that is, if $A$ obeys $c$, then for any monotone approximation $\seq{c_s}$ for $c$, there is a computable approximation $\seq{A_s}$ of $A$ for which the sum above is finite. See Proposition \ref{prop_limit_obedience} below. However, different approximations for $A$ may cause the sum to be infinite. 

Unlike $K$-triviality, strong jump-traceability cannot be characterised by a single cost function; one way to see this is by considering the complexity of the index-set of strong jump-traceability, which is $\Pi^0_4$-complete (Ng \cite{Ng:SJT}). Greenberg and Nies \cite{GreenbergNies:Benign} isolated a class of cost functions which together characterised strong jump-traceability on the c.e.\ sets. Benignity is an effective witness for the limit condition. It is a generalisation of the additive property of the canonical cost function for $K$-triviality. 

Let $\seq{c_s}$ be a monotone approximation for a cost function $c$, and let $\epsilon>0$ be rational. We define an auxiliary sequence of markers $m_1(\epsilon), m_2(\epsilon),\dots$, by letting $m_1(\epsilon) = 0$, and given $m_k(\epsilon)$, letting $m_{k+1}(\epsilon)$ be the least $s>m_k(\epsilon)$ such that $c_{s}(m_k(\epsilon))\ge \epsilon$, if there is such a stage $s$; otherwise, $m_{k+1}(\epsilon)$ is undefined. The fact that $\lim c_s = c$ and that $\lim c = 0$ shows that the sequence $\seq{m_k(\epsilon)}$ must be finite, and so we can let $k(\epsilon) = k_{\seq{c_s}}(\epsilon)$ be the last $k$ such that $m_k(\epsilon)$ is defined. 

\begin{definition}\label{def_benign}
	A cost function $c$ is \emph{benign} if it has a monotone approximation $\seq{c_s}$ for which the function $\epsilon\mapsto k_{\seq{c_s}}(\epsilon)$ is bounded by a computable function. 
\end{definition}

Note that if $\seq{c_s}$ witnesses that $c$ is benign, then the last value $m(\epsilon) = m_{k(\epsilon)}(\epsilon)$ need not be bounded by a computable function; it is $\w$-computably approximable ($\w$-c.e.).

\medskip

Greenberg and Nies showed that a c.e.\ set is strongly jump-traceable if and only if it obeys all benign cost functions. Much like obeying the canonical cost function captures the dynamics of the decanter and golden run methods which are used for working with $K$-trivial oracles, this result shows that benign cost functions capture the dynamics of the box-promotion method when applied to c.e., strongly jump-traceable oracles. 

Greenberg, Hirschfeldt and Nies \cite{GreenbergHirschfeldtNies} showed that every set, not necessarily c.e., which obeys all benign cost functions, must be strongly jump-traceable. In this paper we show that obeying benign cost functions in fact characterises strong jump-traceability on all sets. 

\begin{theorem}\label{thm_benign}
	A set is strongly jump-traceable if and only if it obeys every benign cost function. 
\end{theorem}

The fact that every $K$-trivial set is computable from a c.e.\ one is also deduced using obedience to the canonical cost function. It is easy to see that if a computable approximation $\seq{A_s}$ witnesses that $A$ obeys a cost function $c$, then the associated change-set, which records the changes in this approximation for $A$, is a c.e.\ set which computes $A$ and also obeys the cost function $c$. Hence Theorem \ref{thm_benign} almost gives us Theorem \ref{thm_main_ce}; the connection between benign cost functions and strong jump-traceability established in \cite{GreenbergNies:Benign} shows now that if $A$ is a strongly jump-traceable set, and $h$ is an order function, then there is an $h$-jump-traceable c.e.\ set which computes $A$. (We note that this result implies all the corollaries above). We get Theorem \ref{thm_main_ce} by showing:

\begin{theorem}\label{thm_cost}
	There is a benign cost function $c$ such that for any $\Delta^0_2$ set $A$ obeying $c$, there is a c.e.\ set $W$ computing $A$, which obeys all cost functions that $A$ obeys. 
\end{theorem}

Theorem \ref{thm_main_ce} is an immediate consequence of the conjunction of Theorems \ref{thm_benign} and \ref{thm_cost}. We prove Theorem \ref{thm_benign} in Section \ref{sec_benign} and Theorem \ref{thm_cost} in Section \ref{sec_cost}.

\section{Strongly jump-traceable sets obey benign cost functions} \label{sec_benign}

In this section we prove Theorem \ref{thm_benign}. As we mentioned above, one direction of the theorem is proved in \cite{GreenbergHirschfeldtNies}. For the other direction, we are given a strongly jump-traceable set $A$, and a benign cost function $c$, and show that $A$ obeys $c$.

\subsection{Discussion}

Our departure point is a simplified version of the original argument showing that every strongly jump-traceable set is $\Delta^0_2$. Suppose that we are given a strongly jump-traceable set $A$, and we wish to find a computable approximation $\seq{A_s}$ for $A$. The idea is to \emph{test} binary strings, potential initial segments of $A$. For example, to determine $A(0)$, we try to test both strings $\seq{0}$ and $\seq{1}$, and hopefully get an indication which one is an initial segment of $A$. Our belief about which one may change from time to time, but we need to make sure that it changes only finitely many times, and eventually settles on the correct value. While we fluctuate between $\seq{0}$ and $\seq{1}$, we also test strings of length 2, and match up our guess for which string of length 2 is an initial segment of $A$ with the current guess about which string of length 1 is an initial segment of $A$. Again, our belief about strings of length 2 may change several times, indeed many more than the changes between $\seq{0}$ and $\seq{1}$, but eventually it should settle to the correct value. 

How do we test strings of a given length? We define a functional $\Psi$, fix an order function $h$, which will be designed to grow sufficiently slowly as to enable the combinatorics of the construction, and by the recursion theorem (or by using a universal trace), we have a c.e.\ trace $\seq{T(z)}$ for the partial function $\Psi^A$, bounded by $h$. To test, for example, all strings of a length $\ell$ on some input $z$, we define $\Psi^\s(z)  = \s$ for every string $\s$ of length $\ell$. We then only believe strings which show up in the trace $T(z)$. If $h(z)=1$ then we are done, since only one string may show up in $T(z)$, and the correct string $A\rest{\ell}$ must appear in $T(z)$. However, $h$ must be unbounded,
and once we tested a string $\s$ on some input $z$, we cannot test any extensions of $\s$ on the same input; for the functional $\Gamma$ must be kept consistent. What do we do, then, if $h(z)>1$, and more than one string of length $\ell$ shows up in $T(z)$?

This is where \emph{box promotion} comes into place. Suppose that initially, we use inputs $z$ such that $h(z)=\ell$ to test strings of length $\ell$ (such inputs are sometimes called \emph{$\ell$-boxes}). So when we test strings of length 2, of the four possibilities, we believe at most two. At first, we believe the first string of length 2 which shows up in the relevant trace component, say $\seq{00}$. If another string shows up, say $\seq{01}$, we move to test the length 2 on 1-boxes which we have reserved for this occasion. The reason we can do this is that some 2-boxes have been \emph{promoted}: if $\seq{00}$ is correct, then boxes $z$ for which $\seq{01}\in T(z)$ have spent one of their slots on an incorrect strings. If, for example, later, we believe both $\seq{000}$ and $\seq{001}$, since both have appeared in (the trace for) 3-boxes, then we can use the promoted 2-boxes to decide between the two strings of length 3. After all, neither of these strings extend $\seq{01}$, as $\seq{01}$ has been discovered to be incorrect, and so we can test these strings in the promoted boxes without violating the consistency of $\Gamma$. In general, the promotion mechanism ensures that we have an approximation for $A$ for which there are at most $\ell$ changes in our belief about $A\rest{\ell}$. 

Let $\seq{c_s}$ be a monotone approximation for $c$ which witnesses that $c$ is benign; let $m_k(\epsilon)$ be the associated markers. To construct a computable approximation $\seq{A_s}$ for $A$ for which the sum $\sum_s c_s(A_s)$ is finite, we need, roughly, to give a procedure for guessing initial segments of $A$ such that for all $n$, for all $k\le k(2^{-n})$, the number of changes in our belief about $A\rest{m_k(2^{-n})}$ is (say) $n$. The computable bound on $k(2^{-n})$, the number of lengths we need ``test at level $n$'', allows us to apportion, in advance, sufficiently many $n$-boxes to deal with all of these lengths, even though which lengths are being tested at level $n$ is not known in advance. The fact that the lengths themselves are not known in advance necessitates a first step of ``winnowing'' the strings of new lengths $m_k(2^{-n})$, so that instead of dealing with $2^{m_k(2^{-n})}$ many strings, we are left with at most $n$ such strings. This is done by testing all strings of the given length on an $n$-box reserved for this length, as described above. 

\medskip

As is the case with all box-promotion constructions, the heart of the proof is in the precise combinatorics which tell us which strings are tested on which boxes. One main point is that while we need to prepare $n$-boxes for the possibility that lengths tested at higher levels are promoted all the way down to level $n$, the number of such promotions must be computably bounded in $n$, and cannot rely on the computable bound on $k(2^{-(n+1)})$, $k(2^{-(n+2)}),\dots$. That is, the number of promotions must be tied to the size (or level) of the boxes, and not on the number of lengths that may be tested at that level. 

Consider, for example, the following situation: at some level $n$, we are testing two lengths, $\ell_1$ and $\ell_2$, and tests have returned positively for strings $\s_0$ and $\s_1$ of length $\ell_1$, and strings $\tau_0$ and $\tau_1$ of length $\ell_2$. If, to take an extreme situation for an example, the strings $\s_0,\s_1,\tau_0,\tau_1$ are pairwise incomparable, we could test them \emph{all} on a single input $z$ before we believe them; when we discover which one of them is correct, the other values are certified to be wrong, and give the box $z$ a promotion by three levels. If, on the other hand, $\tau_0$ extends $\s_0$ and $\tau_1$ extends $\s_1$, then we cannot test $\tau_0$ on boxes on which we already tested $\s_0$, and the same holds for $\tau_1$ and $\s_1$. We do not want, though, to let both lengths be promoted (moved to be tested on $(n-1)$-boxes) while $n$-boxes are only promoted by one level (containing only one incorrect value). In this case our action depends on timing:
\begin{itemize}
	\item If $\s_0$ and $\s_1$ appear before $\tau_0$ and $\tau_1$ appear, we promote the length $\ell_1$. We do not promote $\ell_2$, unless another string of length $\ell_2$ appears. If no such new string appears, then our belief about which of $\s_0$ or $\s_1$ is an initial segment of $A$ will dictate which of $\tau_0$ or $\tau_1$ we believe too. 
	\item If $\tau_0$ and $\tau_1$ appear before we see both $\s_0$ and $\s_1$, then we promote the length $\ell_2$. In this case, certainly our belief about which of $\tau_0$ or $\tau_1$ is an initial segment of $A$ would tell us whether to believe $\s_0$ or $\s_1$. 
\end{itemize}
In the first case, an important observation is that if another string $\rho$ of length $\ell_2$ appears, then $\rho$ cannot extend both $\s_0$ and $\s_1$. If $\rho$ does not extend $\s_0$, say, then we can test $\s_0,\rho$ and $\tau_1$ all on one box, and so this box will be eventually promoted by two levels, justifying the promotion of both lengths $\ell_1$ and $\ell_2$ to be tested on $(n-1)$-boxes. Of course, during the construction, we need to test strings on a large number of boxes, to allow for all possible future combinations of sets of strings involving the ones being tested, including strings of future lengths not yet observed.

\subsection{Construction}

As mentioned above, let $\seq{c_s}$ be a monotone approximation for $c$ which witnesses that $c$ is benign; let $m_k(\epsilon)$ be the associated markers. We force these markers to cohere in the following way. For $n<\w$ and $s<\w$ let  
\[ l_s(n) = \max \left(\{n\} \cup \left\{ m_k(2^{-r})  \,:\, r\le n \andd m_k(2^{-r}) \le s \right\}\right).\]
We summarise the properties of the functions $l_s$ in the following lemma. 

\begin{lemma}\label{lem_pro perties_of_l_s} \
	\begin{enumerate}
	\item Each function $l_s$ is non-decreasing, with $n\le l_s(n)\le \max\{n,s\}$.  
	\item For each $n$, the sequence $\seq{l_s(n)}_{s<\w}$ is non-decreasing, and takes finitely many values. Indeed, the function 
	\[ n\mapsto \# \left\{ l_s(n)  \,:\, s<\w \right\}\] is computably bounded.
	\item For all $n$ and $s$, $c_s(l_s(n))< 2^{-n}$. 
	\end{enumerate}
\end{lemma}
 
We fix a computable function $g$ bounding the function $n\mapsto \# \left\{ l_s(n)  \,:\, s<\w \right\}$.

\medskip

For $n\ge 1$, let $\alpha(n) = \binom{n}{0} + \binom{n}{1} + \binom{n}{2}$ be the number of subsets of $\{1,2,\dots, n\}$ of size at most 2. We partition~$\w$ into intervals $M^1$,$I^1$, $M^2$, $I^2, \dots$; the interval~$M^n$ has size $\alpha(n)^{n+g(n)}$ and the interval~$I^n$ has size $n+g(n)$.  We define an order function~$h$ so that $h(x) = n$ for every $x \in M^n \cup I^n$. 

As mentioned, we enumerate a functional $\Psi$. Either by using the recursion theorem (as was done in \cite{CholakDowneyGreenberg:SJT1}) or by using a universal trace (as in \cite{GreenbergNies:Benign}), we obtain a number $o\in \w$ and a c.e.\ trace $T = \seq{T(z)}$ for $\Psi$ which is bounded by $\max\{h,o\}$. 

Each level $n\ge o$ will list an increasing sequence of lengths $\ell_1^n, \ell_2^n,\dots$ which will be tested at level $n$. The list is dynamic -- we may extend it during the construction. However, we will need to ensure that the length of the list is bounded by $n+g(n)$. 

\medskip

The testing of lengths at level $n$ will be in two parts. 

\medskip

\noindent{\textbf{A.}} \emph{Initial testing} of all strings of length $\ell_k^n$ will be performed on a reserved input from the interval $I^n$. We thus enumerate the elements of $I^n$ as $\{z^n_1,z^n_2,\dots, z^n_{n+g(n)}\}$; the input $z^n_k$ is reserved for initial testing of all strings of length $\ell^n_k$. {
We note here that as the list of lengths $\ell^n_1,\ell^n_2,\dots$ may not necessarily reach its maximal length $n+g(n)$, it is possible that some inputs $z^n_k$ will never be used. This is one reason for the fact that $\Psi^A$ will be a \emph{partial function}. In this way we use the full hypothesis of strong jump-traceability of $A$; we cannot hope to make $\Psi^A$ total, and so the proof would not work for merely c.e.\ traceable oracles. }

\medskip

\noindent{\textbf{B.}}	The main bulk of the testing of strings of length $\ell^n_k$ would be performed on inputs from $M^n$. To maximise the interaction between the various lengths (to obtain maximal promotion, we need to test large antichains of strings on inputs from $M^n$), we think of $M^n$ as an $(n+g(n))$-dimensional hypercube, the sides of which each have length $\alpha(n)$. We let $D(n) = \{1,2,\dots, n+g(n)\}$ be the set of ``directions'' (or ``dimensions'') of this hypercube, and use Cartesian coordinates to index the elements of $M^n$ appropriately. The sides of the hypercube are indexed by the subsets of $\{1,2,\dots, n\}$ of size at most 2. So if we let $P(n)$ be the collection of all such subsets, we enumerate the elements of $M^n$ as $z_{\nu} = z^n_\nu$, where $\nu$ ranges over all functions from $D(n)$ to $P(n)$. 

\medskip

The construction begins at stage $o$. At stage $s\ge o$, we act in turn on level $s$, level $s-1,\dots$, down to level $o$. The action at level $n$ consists of: (1) extending the sequence of lengths $\seq{\ell^n_k}$; (2) testing strings on the $n$-cube $M^n$; and (3) if $n>o$, promoting lengths to be tested on level $n-1$. 

\medskip

Let $n\in [o,\dots, s]$. The action at level $n$ at stage $s$ is as follows:

\medskip

\noindent{\textbf{1.}} If $n<s$ and some lengths have just been promoted from level $n+1$, we append them to the list of lengths $\ell^n_1,\ell^n_2,\dots$ tested at level $n$, ordered by magnitude (we will make sure that the promoted lengths are longer than lengths already tested at level $n$). 

If $l_s(n)$ is greater than the lengths currently tested at level $n$ (including the lengths which have just been promoted), we add it too to the list of lengths tested at level $n$. 

We are assuming now that at every stage, the number of lengths tested at level $n$ is at most $n+g(n)$. We will prove this later (Section \ref{subsec_justify}).

For each length $\ell^n_k$ which was added to the list, we test all strings of length $\ell^n_k$ on $z^n_k$. This means we define $\Psi^\s(z^n_k) = \s$ for every string $\s$ of length $\ell^n_k$.

\medskip

\noindent{\textbf{2.}} 
Suppose that $\ell^n_k$ is defined at stage $s$. We list the elements of $T(z^n_k)$ by $\s^n_k(1), \s^n_k(2),\dots$ as they appear. Because $n\ge o$ and $z^n_k\in I^n$, we have $|T(z^n_k)|\le n$, so the list has length at most $n$. 

Suppose that $\s^n_k(i)$ has appeared in $T(z^n_k)$. Recall that $P(n)$ is the collection of subsets of $\{1,2,\dots, n\}$ of size at most 2. For every $\nu\colon D(n)\to P(n)$ such that $i\in \nu(k)$, we test $\s^n_k(i)$ on $z_\nu = z^n_\nu$. Fix such $\nu$. We need to ensure that $\Psi$ remains consistent; the point is that there may be strings comparable with $\s = \s^n_k(i)$ which are already tested on $z_\nu$. To test $\s$ on $z_\nu$ while keeping $\Psi$ consistent, we define $\Psi^\tau(z_\nu) = \tau$ for every extension $\tau$ of $\s$ of length $s$ which does not extend any string already tested on $z_\nu$. 

Using other notation, we let $Z_{\nu,s}$ be the collection of strings $\rho$ for which we defined $\Psi^\rho(z_\nu)=\rho$ by the end of stage $s$, and let $\ZZ_{\nu,s} = [Z_{\nu,s}]$ be the clopen subset of Cantor space $2^\w$ determined by the set of strings $Z_{\nu,s}$ (the collection of all infinite extensions of strings in $Z_{\nu,s}$). Testing a string $\s$ on $z_{\nu}$ at stage $s$ means adding strings of length $s$ to $Z_{\nu,s-1}$ so as to keep $Z_{\nu,s}$ an antichain, but ensuring that $[\s]\subseteq \ZZ_{\nu,s}$. 

\medskip

\noindent{\textbf{3.}} 
For $\nu\colon D(n)\to P(n)$, we may assume that $T_s(z_\nu)\subseteq Z_{\nu,s}$. (Otherwise, simply ignore all other values, acting as though $T_s(z_\nu)$ were replaced by $T_s(z_\nu) \cap Z_{\nu,s}$.) We let $\TT_s(z_\nu) = [T_s(z_\nu)]$ be the clopen subset of Cantor space determined by $T(z_\nu)$. 

Let $k\le n+g(n)$ such that $\ell^n_k$ is defined by stage $s$, and let $i\le n$ such that $\s^n_k(i)$ is defined by stage $s$, that is, $T(z^n_k)$ already contains at least $i$ many elements by stage $s$. The test of $\s^n_k(i)$ is \emph{successful} if for all $\nu$ such that $i\in \nu(k)$, that is, for all $\nu$ such that $\s$ was tested on $z_\nu$, we have $[\s]\cap \TT_s(z_\nu)\ne \emptyset$. In other words, if some string which is comparable with $\s$ appears in $T(z_\nu)$ by stage $s$. 

For the purpose of the following definition, let $\ell^n_0 = 0$. We say there is a \emph{conflict} at length $\ell^n_k$ (and level $n$) if there are two strings $\s_0 = \s^n_k(i)$ and $\s_1 = \s^n_k(j)$ of length $\ell^n_k$, both of whose tests are successful by stage $s$, such that $\s_0\rest{\ell^n_{k-1}} = \s_1\rest{\ell^n_{k-1}}$. We note, for future reference, that if there is a conflict at length $\ell^n_k$ at stage $s$, then this conflict persists at every later stage. 

At stage $s$, if $n>o$, then we promote to level $n-1$ all lengths $\ell^n_k$ for which there is a conflict at stage $s$, and which are longer than any length already tested at level $n-1$. 

\medskip

These instructions determine our action for level $n$ at stage $s$, and so completely describe the construction.

\subsection{Justification} \label{subsec_justify}

Before we show how the construction gives us the desired approximation for $A$, we first need to show that we can actually implement the construction. We need to prove that we have allocated sufficiently many $n$-boxes to each level $n$; that is, we must show that the list of lengths $\seq{\ell^n_k}$ tested at level $n$ has length at most $n+g(n)$. 

For $n\ge o$ and $s<\w$, let $k_s(n)$ be the number of lengths tested at level $n$ by the end of stage $s$. That is, at the end of stage $s$, the lengths $\ell^n_k$ are defined for $k\le k_s(n)$. We need to show that for all $s$, $k_s(n)\le n+g(n)$. 

There are two streams contributing lengths to test at level $n$: lengths promoted from level $n+1$, and lengths of the form $l_s(n)$. Of the latter, there are at most $g(n)$ many. Hence, it remains to show that there are at most $n$ many lengths that are promoted by level $n+1$. Shifting indices, we show that level $n$ promotes at most $n-1$ many lengths. 

Indeed, we show the following:

\begin{lemma}\label{lem_bound_on_conflicts}
	Let $n\ge o$ and let $s\ge o$. Then there are at most $n-1$ many lengths $\ell^n_k$ at which there is a conflict (for level $n$) at stage $s$. 
\end{lemma}

To prove Lemma \ref{lem_bound_on_conflicts}, fix $n\ge o$ and $s\ge o$. Let $N$ be the number of lengths at which there is a conflict (at level $n$) at the end of stage $s$. We show that there is some $\nu\colon D(n)\to P(n)$ such that $|T_s(z_\nu)|-1\ge N$. Using the fact that $n\ge o$ and $z_\nu\in M^n$ we see that $|T_s(z_\nu)|\le n$, which establishes the desired bound. 

In order to define $\nu$, we define an increasing sequence of antichains of strings, indexed in reverse $C_{k_s(n)+1} \subseteq C_{k_s(n)}\subseteq C_{k_s(n)-1}\subseteq \cdots \subseteq C_1$, starting with $C_{k_s(n)+1} = \emptyset$. Each set $C_k$ consists of strings of lengths $\ell^n_{k'}$ for $k'\ge k$. Let $k\in \{1,\dots, k_s(n)\}$; we assume that $C_{k+1}$ has been defined, and we show how to define $C_k$.

The definition is split into two cases. First, suppose that there is no conflict at stage $s$ in length $\ell^n_k$. We then let $C_k =C_{k+1}$ and $\nu(k) = \emptyset$. 

We assume then that there is a conflict in length $\ell^n_k$ at stage $s$. Let $\s_0 = \s^n_k(i)$ and $\s_1 = \s^n_k(j)$ be a pair witnessing this conflict. We let $C_{k}$ be a maximal antichain from $C_{k+1}\cup\{\s_0,\s_1\}$ containing $C_{k+1}$. In other words, if neither $\s_0$ nor $\s_1$ are comparable with any string in $C_{k+1}$, then we let $C_k = C_{k+1}\cup\{\s_0,\s_1\}$; otherwise, if either $\s_0$ or $\s_1$ is incomparable with all the strings in $C_{k+1}$, then we let $C_k$ be one of $C_{k+1}\cup \{\s_0\}$ or $C_{k+1}\cup \{\s_1\}$, making sure that we choose so that $C_k$ is an antichain; and finally, if both $\s_0$ and $\s_1$ are comparable with strings in $C_{k+1}$, then we let $C_k = C_{k+1}$. 

Now given the sequence of sets $C_k$, we can define the index function $\nu$:
\begin{itemize}
	\item For $k\in \{1,2,\dots, k_s(n)\}$, we let 
	\[ \nu(k) = \left\{ i\le n  \,:\, \s^n_k(i)\in C_k \right\}.\]
	\item For $k\in \{k_s(n)+1,\dots, n+g(n)\}$, we let $\nu(k) = \emptyset$. 
\end{itemize}
Since the strings in $C_k$ of length $\ell^n_k$ are precisely the strings in $C_k\setminus C_{k+1}$, we see that for all $k$, $\nu(k)$ is indeed a set of size at most 2, so $\nu$ is a function from $D(n)$ to $P(n)$. The point of this definition is that the strings tested on $z_\nu$ are precisely the strings in $C_1$.

\medskip

Letting $\ell^n_0 = 0$ again, for $k\in \{1,\dots, k_s(n)+1\}$, let 
\[ D_k = \left\{ \s\rest{\ell^n_{k-1}}  \,:\, \s\in C_k \right\} ,\]
and let 
\[ p_k  = |C_k| - |D_k|.\]
Note that $p_{k_s(n)+1}=0$, and that unless $C_1$ is empty, $|C_1| = p_1+1$. 

\begin{claim}
	For all $k\le k_s(n)$, $p_{k} \ge p_{k+1}$.
\end{claim}

\begin{proof}
	For every string $\tau$ in $D_k$ which has no extension in $D_{k+1}$, there is an extension $\s$ of $\tau$ in $C_k\setminus C_{k+1}$. Therefore, $$p_k - p_{k+1} = |C_k|-|C_{k+1}|+|D_{k+1}|-|D_k| \geq 0.$$
\end{proof}

\begin{claim}
	If $\ell^n_k$ has a conflict at stage $s$, then $p_k > p_{k+1}$. 
\end{claim}

\begin{proof}
	Let $\s_0$ and $\s_1$ be the strings that were chosen at step $k$ to witness that $\ell^n_k$ has a conflict at stage $s$. By definition of having a conflict, $\s_0 \rest{\ell^n_{k-1}} = \s_1\rest{\ell^n_{k-1}}$; we let $\tau$ denote this string. 
	
	There are three cases. In all three cases, we note that every string in $D_k$ other than possibly $\tau$ has an extension in $D_{k+1}$. 
	
	\medskip
	
	If $C_k = C_{k+1} \cup \{\s_0,\s_1\}$ then we need to show that $|D_k|\le |D_{k+1}| +1$, which follows from the fact we just mentioned, that every string in $D_k$ other than $\tau$ has an extension in $D_{k+1}$. 
	
	\medskip
	
	In the second case, we assume that $C_k$ is obtained from $C_{k+1}$ by adding one string, say $\s_0$; we need to show that $|D_k|\le |D_{k+1}|$. But $\s_1$ is comparable with some string in $C_{k+1}$, and in fact must be extended by some string in $C_{k+1}$. Hence $\s_1\in D_{k+1}$, i.e.\ $\tau$ is extended by some string in $D_{k+1}$, and therefore every string in $D_k$ is extended by some string in $D_{k+1}$. 
	
	\medskip
	
	Finally, suppose that $C_{k+1} = C_k$; we need to show that $|D_{k+1}|\le |D_k|-1$. Since both $\s_0$ and $\s_1$ are comparable with elements of $C_k$, both are elements of $D_{k+1}$, and so $\tau$ has two extensions in $D_{k+1}$, while every other string in $D_k$ has an extension in $D_{k+1}$. 	
\end{proof}

Hence $p_1\ge N$. If $C_1$ is empty, then $N=0$, so we may assume that $C_1$ is nonempty, and so $|C_1| = p_1 + 1$ is at least one more than $N$. Then Lemma \ref{lem_bound_on_conflicts}, and with it our justification for the construction, is completed once we establish the following claim.

\begin{claim}
	$|T_s(z_\nu)| \ge |C_1|$.
\end{claim}

\begin{proof}
	We show that $T_s(z_\nu)$ contains only strings which are extensions of strings in $C_1$, and that each string in $C_1$ has an extension in $T_s(z_\nu)$. 
	
	Recall that we let $Z_{\nu,s}$ be the collection of strings that were actually tested on $z_{\nu}$ by stage $s$, that, is, the collection of strings $\rho$ for which we defined $\Psi^\rho(z_\nu) = \rho$ by the end of stage $s$. 
	
	Our instructions (and the definition of $\nu$) say that the strings tested on $z_\nu$ are precisely the strings in $C_1$. Since $C_1$ is an antichain, this means that before some string $\s$ is tested on $z_\nu$, we have $[\s]\cap \ZZ_{\nu,t} = \emptyset$, and so when testing $\s$, we only add extensions of $\s$ to $Z_{\nu,s}$. Since we assumed that $T_s(z_\nu)\subseteq Z_{\nu,s}$, we see that all strings in $T_s(z_\nu)$ are extensions of strings in $C_1$. 
	
	Let $\s\in C_1$. Then $\s = \s^n_k(i)$ for some $k$ and $i$ is part of a pair of strings witnessing that there is a conflict at length $\ell^n_k$ (and level $n$) by stage $s$. So the test of $\s$ on $M^n$ is successful by the end of stage $s$. Since $\s$ is tested on $z_\nu$, we have $[\s]\cap \TT_s(z_\nu)\ne\emptyset$. Since no proper initial segment of $\s$ is tested on $z_\nu$, this means that some extension of $\s$ is an element of $T_s(z_\nu)$. 
\end{proof}
	
\medskip

\subsection{The approximation of $A$}

We now show how to find a computable approximation for $A$ witnessing that $A$ obeys $c$. 

\medskip

For $n\ge o$, let $k(n) = \lim_s k_s(n)$ be the number of lengths ever tested at level $n$.

\begin{lemma}\label{lem_correct_value_appears}
	For all $n\ge o$ and all $k\le k(n)$, The string $A\rest{\ell^n_k}$ is eventually successfully tested at level $n$.
\end{lemma}

\begin{proof}
	Let $s_0$ be the stage at which the length $\ell^n_k$ is first tested at level $n$. Let $\rho = A\rest{\ell^n_k}$. At stage $s_0$, we define $\Psi^\rho(z^n_k) = \rho$, and so $\Psi^A(z^n_k) = \rho$. Since $T$ traces $\Psi^A$, we have $\rho\in T(z^n_k)$; this is discovered by some stage $s_1> s_0$. At stage $s_1$ we test $\rho$ on elements $z_\nu$ of $M^n$. Fix such an input $z_\nu$. We need to show that $[\rho]\cap \TT(z_\nu)$ is nonempty. 
	
	At stage $s_1$, we enumerate strings into $Z_{\nu}$ to ensure that $[\rho]\subseteq \ZZ_\nu$. Hence $A\in \ZZ_\nu$, in other words, $z_\nu\in \dom \Psi^A$. Since $T$ traces $\Psi^A$, we have $\Psi^A(z_\nu)\in T(z_\nu)$. All axioms of $\Psi$ are of the form $\Psi^\tau(z)= \tau$ for binary strings $\tau$, so $\tau = \Psi^A(z_\nu)$ is an initial segment of $A$, and so is comparable with $\rho$. Then $[\tau]\subseteq \TT(z_\nu)$ implies that $[\rho]\cap \TT(z_\nu)\ne \emptyset$. 
\end{proof}

For $n\in [o,\dots, s]$, let $\ell^n[s] = \ell^n_{k_s(n)}$ be the longest length tested at level $n$ at the end of stage $s$. Then for all $s\ge o$, $\ell^o[s]\le \ell^{o+1}[s]\le \cdots \le \ell^s[s] = s$, because if we let $\ell^n[s] = l_s(n)$ at stage $s$, then (Lemma \ref{lem_pro perties_of_l_s}) $l_s(n+1)\ge l_s(n)$ and so we define $\ell^{n+1}[s] = l_s(n+1)$ if this length is longer than previous lengths tested at level $n+1$. Also, since at stage $s$ we test $s = l_s(s)$ at level $s$, we see that for all $s\ge n$, $\ell^n[s]\ge n$. 

For $n\ge o$, we let $\ell^n =  \ell^n_{k(n)} = \lim_s \ell^n[s]$ be the longest length ever tested at level $n$. Let $\rho^* = A\rest {\ell^o}$. Let $s_o>o$ be a stage sufficiently late so that $\ell^{o}[s_o]=\ell^o$ and the string $\rho^*$ is successfully tested at level $o$ by stage $s_o$. 

We note that other than specifying $\rho^*$, the construction is uniform (in the computable index for $\seq{c_s}$). The reason for the nonuniform aspect of the construction is the overhead $o$ charged by the recursion theorem; if we had access to 1-boxes, the construction would be completely uniform.

\medskip

Let $s\ge s_o$ and $n\ge o$. A string $\s$ of length $\ell^n[s]$ is \emph{$n$-believable} at stage $s$ if:
\begin{itemize}
	\item $\s$ extends $\rho^*$; and
	\item for all $m\in [o,n]$, and for all $k\le k_s(m)$, the string $\s\rest {\ell^m_k}$ is successfully tested at level $m$ by the end of stage $s$. 
\end{itemize}

Lemma \ref{lem_correct_value_appears} shows that for all $n$, the string $A\rest{\ell^n}$ is $n$-believable at almost every stage. 

\begin{claim}\label{clm_unique_believability}
	Let $s\ge s_o$. For every $n\ge o$, there is at most one string which is $n$-believable at stage $s$. 
\end{claim}

\begin{proof}
	By induction on $n$. For $n=o$ this is clear, because $\rho^*$ has length $\ell^o[s]$ for all $s\ge s_o$. 
	
	Let $n>o$, and suppose that there is at most one string which is $(n-1)$-believable at stage $s$. Suppose, for contradiction, that there are two strings $\tau_0$ and $\tau_1$ which are both $n$-believable at stage $s$. Then both $\tau_0\rest{\ell^{n-1}[s]}$ and $\tau_1\rest{\ell^{n-1}[s]}$ are $(n-1)$-believable at stage $s$, and so are equal. Let $k$ be the least index such that $\tau_0\rest {\ell^n_k} \ne \tau_1\rest{\ell^n_k}$. Of course $k$ exists, since $\tau_0\ne\tau_1$ are both of length $\ell^n_{k_s(n)}$, and $\ell^n_k> \ell^{n-1}[s]$. In other words, $\ell^n_k$ is longer than any length tested at level $n-1$ at stage $s$. But then the strings $\tau_0\rest{\ell^n_k}$ and $\tau_1\rest{\ell^n_k}$ witness that there is a conflict at length $\ell^n_k$ at stage $s$, and so we would promote $\ell^n_k$ to be tested at level $n-1$ by the end stage $s$, contradicting the assumption that $\ell^n_k$ is not tested at level $n-1$ at stage $s$. 
\end{proof}

We can now define the computable approximation for $A$. We define a computable sequence of \emph{stages}: the stage $s_o$ has been defined above; we may assume that $s_o\ge o+1$. For $t>o$, given $s_{t-1}$, we define $s_{t}$ to be the least stage $s>s_{t-1}$ at which there is a $t$-believable string $\s_t$. So $s_{t-1}\ge t$. We let $A_t = \conc{\s_t}{0^\w}$. The fact that $A\rest{\ell^n}$ is $n$-believable at almost every stage (and that $\ell^n\ge n$) implies that $\lim_t A_t = A$. 

\medskip

For $t\ge o$, let $x_t$ be the least number $x$ such that $A_t(x)\ne A_{t-1}(x)$. It remains to show that $\sum_{t>o} c_t(x_t)$ is finite. For all $n\ge 0$, let 
\[ S_n = \left\{ t> o\,:\, c_t(x_t) \ge 2^{-n} \right\}.\]
Then $\sum c_t(x_t)<\infty$ will follow from any polynomial bound on $|S_n|$. Let $n>o$, and let $t\in S_n$. Let $s = s_{t-1}$, and let $\bar s = s_t$.  Since $t\le s$ and $\seq{c_s}$ is monotone, we have $c_{s}(x_t)\ge c_t(x_t) \ge 2^{-n}$. Since $c_s(l_s(n))<2^{-n}$ (Lemma \ref{lem_pro perties_of_l_s}), and the function $c_s$ is monotone, we have $x_t<l_s(n)$. So $A_t\rest{l_s(n)}\ne A_{t-1}\rest{l_s(n)}$. 

Suppose that $t>n$. Then the strings $\s_t$ and $\s_{t-1}$ are at least as long as $\ell^{t-1}[s]$ which is not smaller than $\ell^n[s]$, which in turn is not smaller than $l_s(n)$, by the instruction for testing $l_s(n)$ at level $n$ at stage $s$ if it is a large number. So we actually have $\s_t\rest{\ell^n[s]}\ne \s_{t-1}\rest{\ell^n[s]}$. 

Let $m\le n$ be the least such that $\s_t\rest{\ell^m[s]}\ne \s_{t-1}\rest{\ell^m[s]}$; since both $\s_t$ and $\s_{t-1}$ extend $\rho^*$ we have $m>o$. Let $k\le k_s(n)$ be the least such that $\s_t\rest{\ell^m_k}\ne \s_{t-1}\rest{\ell^m_k}$; the minimality of $m$ implies that $\ell^m_k > \ell^{m-1}[s]$. 

Let $\tau_0 = \s_{t-1}\rest{\ell^m_k}$ and $\tau_1 = \s_t\rest{\ell^m_k}$. So $\tau_0$ and $\tau_1$ are distinct. Since $\s_{t-1}$ is $(t-1)$-believable at stage $s$, and $m\le n\le t-1$, the string $\tau_0$ is successfully tested at level $m$ by stage $s$, and similarly, $\tau_1$ is successfully tested at level $m$ by stage $\bar s$. Thus there is a conflict at length $\ell^m_k$ at stage $\bar s$,  which implies that $\ell^{m-1}[\bar s]\ge \ell^m_k$. We observed that $\ell^m_k > \ell^{m-1}[s]$, and so there is no conflict at level $\ell^m_k$ at stage $s$. 

This means that if $t$ and $u$ are two stages in $S_n$, and $u>t>n$, then there is some $m\le n$ and some length $\ell = \ell^m_k$ at which there is no conflict at stage $s_{t-1}$ but there is a conflict at stage $s_t \le s_{u-1}$. Lemma \ref{lem_bound_on_conflicts} states this can happen, for each $m$, at most $m-1$ times, and so overall, there are at most $\binom{n}{2}$ many stages greater than $n$ in $S_n$; that is, $|S_n|\le n + \binom{n}2$. This gives a polynomial bound on $|S_n|$ and completes the proof. 
\begin{flushright}\qedsymbol\end{flushright}

\section{A c.e.\ set computing a given set} \label{sec_cost}

In this section we give a proof of Theorem \ref{thm_cost}: we construct a benign cost function $c$ such that for any $\Delta^0_2$ set $A$ obeying $c$, there is a c.e.\ set $W$ computing $A$ which obeys all cost functions that $A$ obeys. 

%

\subsection{A simplification}

Even though the cost function $c$ works for any $\Delta^0_2$ set $A$, we may assume that we are given a particular computable approximation $\seq{A_s}$ to a $\Delta^0_2$ set $A$ which obeys $c$, and define $c$ using the approximation. 

To see why this seemingly circular construction is in fact legal, we enumerate as $\seq{\seq{A^k_s}_{s<\w}}_{k<\w}$ all partial sequences of uniformly computable functions; we think of $\seq{A^k_s}_{s<\w}$ as the $k\tth$ potential computable approximation for a $\Delta^0_2$ set $A^k$. 


For each $k<\w$, we define a benign cost function $c^k$, together with a monotone approximation $\seq{c^k_s}$ for $c_k$ and a computable function $g^k$ which together witness that $c^k$ is benign; all of these, uniformly in $k$. The important dictum is: \emph{even if $\seq{A^k_s}$ is not total, we must make $\seq{c^k_s}$ and $g^k$ total.} We ensure that $c^k(x)\le 1$ for all $x$ and $k$. 

Once these are constructed, we let $c = \sum_{k<\w} 2^{-k}c^k$. 

\begin{lemma}\label{lem_sum_of_benign}
	$c$ is a benign cost function. 
\end{lemma}

\begin{proof}
	For $s<\w$ let $c_s(x) = \sum_{k < s}2^{-k}c^k_s(x)$. Then $\seq{c_s}$ is a monotone approximation of $c$. For benignity, the point is that since $c^k\le 1$, only finitely many $c^k$ can contribute more that $\epsilon$ to $c$. We note that for all $k$, if $\II$ is a set of disjoint intervals of $\w$ such that for all $[x,s)\in \II$ we have $c^k_s(x)\ge \epsilon$, then $|\II|\le g^k(\epsilon)$. Fix $\epsilon>0$, and let $m_1(\epsilon)$, $m_2(\epsilon), \cdots$, $m_{k(\epsilon)}(\epsilon)$ be the sequence of markers associated with $\seq{c_s}$. Let $\II$ be the set of intervals $[m_i(\epsilon), m_{i+1}(\epsilon))$ for $i< k(\epsilon)$. For all $[x,s)\in \II$ we have 
	\[ \epsilon \le c_s(x) = \sum_{k<s} 2^{-k}c_s^k(x) .\] 
Let $K = -\log_2(\epsilon)+1$. Since $c_s^k(x)\le 1$ for all $k$, we have 
\[ \sum_{k>K} 2^{-k}c_s^k(x) \le \frac{\epsilon}2.\]
It follows that for some $k\le K$ we have $c_s^k(x)\ge \epsilon/4$. Hence 
\[ k(\epsilon) \le \sum_{k\le K} g^k\left(\frac{\epsilon}4 \right) \]
which is a computable bound on $k(\epsilon)$. 	
\end{proof}

Suppose that a $\Delta^0_2$ set $A$ obeys $c$. By Nies's result from \cite{Nies:calculus}, which is repeated as Proposition \ref{prop_limit_obedience} below, there is a computable approximation $\seq{A_s}$ for $A$ such that $\sum c_s(A_s)\le 1$. This means that for all $k$, $\sum c^k_s(A_s)\le 2^k$; in particular for $k$ such that $\seq{A_s} = \seq{A^k_s}$. Thus, it suffices to construct $\seq{c_s^k}$ and $g^k$, uniformly in $k$, such that if $\seq{A^k_s}$ is indeed a $\Delta^0_2$ approximation for a set $A$, and $\sum c_s^k(A_s^k)\le 2^k$, then there is some c.e.\ set $W$ computing $A$ which obeys all cost functions that $A$ obeys. The construction for each $k$ is independent. 

In the sequel, we omit the index $k$; we assume that we are given a partial sequence $\seq{A_s}$ and construct total $\seq{c_s}$ and $g$ with the desired property. Although we have to make $\seq{c_s}$ and $g$ total and $c_s$ bounded by 1, regardless of the partiality of $\seq{A_s}$, we note that unless $\seq{A_s}$ is total and is a computable approximation for a set $A$, and $\sum c_s(A_s)\le 2^k$, then the construction of $W$ need not be total.

\subsection{More on cost functions}	

Given the approximation $\seq{A_s}$ for $A$, we need to test whether $A$ obeys a given cost function $d$, with a given approximation $\seq{d_s}$. But of course it is possible that $\sum d_s(A_s)$ is infinite, while some other approximation for $A$ witnesses that $A$ obeys $d$. Any other approximation can be compared with the given approximation $\seq{A_s}$, and so it suffices to examine a speed-up of the given approximation. 

Further, it suffices to test cost functions bounded by 1. This is all ensured by the following proposition. A version of this proposition appears in \cite{Nies:calculus}, but as we give it in slightly different form, we give a full proof here for completeness.

\begin{proposition}\label{prop_limit_obedience}
	Let $B$ be a $\Delta^0_2$ set which obeys a cost function $d$. For any monotone approximation $\seq{d_s}$ of $d$, there is a computable approximation $\seq{B_s}$ of $B$ such that $\sum d_s(B_s)$ is finite. 
	
	Moreover, if $\seq{B_s}$ is a given computable approximation of $B$, then there is an increasing computable function $h$ such that $\sum d_s(B_{h(s)})\le 1$. 
\end{proposition}

\begin{proof}
	Fix a computable approximation~$\seq{B_s}$.  It is sufficient to find a computable function $h$ such that $\sum_s d_s(B_{h(s)})$ is finite; we can then decrease the sum by any finite amount by omitting finitely many initial stages. 	
	
	\medskip
	
	Let $\seq{e_s}$ be a monotone approximation of $d$ and $\langle \hat B_s \rangle$ be a computable approximation of $B$ such that $\sum e_s(\hat B_s)$ is finite. We define increasing sequences $\seq{t_s}$ of stages and $\seq{x_s}$ of numbers (lengths) as follows. We let $t_{-1}=x_{-1}=1$. For $s\ge 0$, given $t_{s-1}$ and $x_{s-1}$, we search for a pair $(t,x)$ such that $t>t_{s-1}$, $x>x_{s-1}$ and
	\begin{itemize}
		\item $d_t(x)< 2^{-(s+1)}$;
		\item $B_t\rest{x} = \hat B_t\rest x$; and
		\item For all $y< x_{s-1}$, $2e_t(y) \ge d_t(y)$. 
	\end{itemize}
	Such a pair $(t,x)$ exists because $\lim e_s = \lim d_s$, $\lim \hat B_s = \lim B_s$, and $\lim_x d(x)=0$ (the limit condition for $d$). We let $(t_s,x_s)$ be the least such pair that we find. We note that for all $s>0$, $s\le t_{s-1}$. We now let $h(s) = t_{s+1}$ for all $s\ge 0$.
	
	\medskip
	
	We claim that 
	
	\[ \sum d_s( B_{h(s)}) \le 2\sum e_s(\hat B_s) + \sum_s 2^{-s},\]
	which is finite. For let $s>0$, and let $y_s$ be the least number $y$ such that $B_{h(s)}(y)\ne B_{h(s-1)}(y)$; so $\sum d_s(B_{h(s)}) = \sum d_s(y_s)$. There are two cases. 
	
	If $y_s\ge x_{s-1}$, then $d_{t_{s-1}}(y_s)< 2^{-s}$, and by monotony, 
	since $t_{s-1}\ge s$, we have $d_s(y_s) < 2^{-s}$. 
	
	In the second case, we have $y_s < x_s, x_{s+1}$, and so $B_{h(s-1)}(y_s) = B_{t_s}(y_s) = \hat B_{t_s}(y_s)$ 
	and $B_{h(s)}(y_s) = B_{t_{s+1}}(y_s) = \hat B_{t_{s+1}}(y_s)$. Hence $\hat B_{t_{s+1}}(y_s)\ne \hat B_{t_s}(y_s)$. 
	There is some stage $t\in (t_s,t_{s+1}]$ such that $\hat B_t(y_s)\ne \hat B_{t-1}(y_s)$. 
	Let $z$ be the least number such that $\hat B_t(z)\ne \hat B_{t-1}(z)$. Since $s\le t_s$, we have 
	$d_s(y_s)\le d_{t_s}(y_s)$. Since $y_s<x_{s-1}$, we have $d_{t_s}(y_s)\le 2e_{t_s} (y_s)$. Again by 
	monotony, we have $e_{t_s}(y_s)\le e_t(y_s)$. And since $z\le y_s$, we have $e_t(y_s)\le e_t(z)$. Overall, we get
	\[ d_s(y_s) \le 2 e_t(z) ,\]
	and $e_t(z)$ is a summand in $\sum e_s(\hat B_s)$, which is counted only against $s$, as $h(s-1)< t< h(s)$. 	
\end{proof}

\medskip

To ensure that the set $W$ obeys every cost function that $A$ obeys, we need to monitor all possible cost functions. So we need to list them: we need to show that they are uniformly $\Delta^0_2$, indeed with uniformly computable monotone approximations. This cannot be done effectively, because the limit condition cannot be determined in a $\Delta^0_2$ fashion. However, we will not need the limit condition during the construction, only during the verification, and so we list monotone cost functions which possibly fail the limit condition. 

\begin{lemma}\label{lem_listing_cost_functions}
	There is a list $\seq{d^e}_{e<\w}$ of all monotone cost functions (which possibly fail the limit condition) bounded by 1, such that from an index $e$ we can effectively obtain a monotone approximation $\seq{d^e_s}_{s<\w}$ for $d^e$. We may assume that $d^e_s\le 1$ for all $e$ and $s$, and that for all $e$, $s$ and $x\ge s$ we have $d^e_s(x)=0$. 
\end{lemma}

\begin{proof}
	The idea is delaying. In this proof we do not assume that cost functions satisfy the limit condition, but we do assume that they are total. We need to show that given a partial uniformly computable sequence $\seq{d_s}$ we can produce, uniformly, a total monotone approximation $\langle{\hat d_s}\rangle$ of a cost function $\hat d$ such that if $\seq{d_s}$ is a monotone approximation of a cost function $d$ bounded by 1, then $\hat d = d$. To do this, while keeping monotony, for every $s<\w$ we let $t(s)\le s$ be the greatest $t\le s$ such that after calculating for $s$ steps, we see $d_u(x)$ converge for all pairs $(u,x)$ such that $u\le t$ and $x\le t$, each value $d_u(x)$ is bounded by 1, and the array $\seq{d_u(x)}_{u,x\le t}$ is monotone (non-increasing in $x$ and non-decreasing in $u$). We let $\hat d_s(x) = d_{t(s)}(x)$ for all $x\le t(s)$, and $\hat d_s(x) = 0$ for all $x> t(s)$. 
\end{proof}

\subsection{Discussion}

Returning to our construction, recall that we are given a partial approximation $\seq{A_s}$ and a constant $k$, and need to produce a (total) monotone approximation $\seq{c_s}$ of a cost function $c$ and a computable function $g$ witnessing that $c$ is benign; and we need to ensure that if $\seq{A_s}$ is a total approximation of a $\Delta^0_2$ set $A$ and $\sum c_s(A_s)\le 2^k$, then there is a c.e.\ set $W$ computing $A$ which obeys every cost function that $B$ obeys. 

The main tool we use is that of a \emph{change set}. For any computable approximation $\seq{B_s}$ of a $\Delta^0_2$ set $B$, the associated \emph{change} set $W(\seq{B_s})$ consists of the pairs $(x,n)$ such that there are at least $n$ many stages $s$ such that $B_{s+1}(x)\ne B_s(x)$. The obvious enumeration $\seq{W_s}$ of $W$ enumerates a pair $(x,n)$ into $W_s$ if there are at least $n$ many stages $t<s$ such that $B_{t+1}(x)\ne B_t(x)$. It is immediate that the change set is c.e.\ and computes $B$. It is also not hard to show that for any monotone approximation $\seq{d_s}$ for a cost function we have 
\[ \sum d_s(W_s)  \le \sum d_s(B_s),\]
and so if $\seq{B_s}$ witnesses that $B$ obeys $d = \lim d_s$, then $\seq{W_s}$ witnesses that $W$ obeys $d$ as well. Nies used this argument to show that every $K$-trivial set is computable from a c.e.\ $K$-trivial set. 

Thus if $A = \lim A_s$ (if it exists) obeys some cost function $d$, we immediately get a c.e.\ set computing $A$ which also obeys $d$. The difficulty arises when we consider more than one cost function. The point is that different cost functions obeyed by $B$ would require \emph{faster} enumerations of $B$, and the associated change sets may have distinct Turing degrees. In general, it is not the case that the change set for a given enumeration of a $\Delta^0_2$ set $B$ would obey all cost functions obeyed by $B$. For an extreme example, it is not difficult to devise a computable approximation for the empty set for which the associated change set is Turing complete. The point is that a faster approximation of a $\Delta^0_2$ set may undo changes to some input $B(x)$, whereas the change set for the original approximation must record the change to $B(x)$ (and also its undoing), and must pay costs associated with such recordings. 

The idea of our construction is to let $W$ be the change set of some speed-up of the approximation $\seq{A_s}$. We define an increasing partial computable function $f$. If $\seq{A_s}$ is total, approximates $A$, and $\sum c_s(A_s)\le 2^k$, then $f$ will be total, and we will let $W$ be the change set of the approximation $\seq{A_{f(s)}}$. Roughly, the role of $f$ would be to ensure that not too may undone changes in some $A(x)$ would be recorded by $W$ and associated costs paid. To be more precise, we discuss our requirements in detail.

\medskip 

Let $\seq{d^i}_{i<\w}$ be a list of cost functions (possibly failing the limit condition) bounded by 1, as given by Lemma \ref{lem_listing_cost_functions} (with associated approximations $\seq{d^i_s}$), and let $\seq{h^j}_{j<\w}$ be an effective list of all partial computable functions whose domain is an initial segment of $\w$ and which are strictly increasing on their domain. To save indices (we are into the whole brevity thing), 
 we renumber the list of pairs $\seq{d^i,h^j}_{i,j<\w}$ as $\seq{d^e,h^e}_{e<\w}$. 

Let $e<\w$. The requirement $S^e$ states that if $h^e$ is total and $\sum d^e_s(A_{(f\circ h^e)(s)})\le \nobreak 1$, then there is some total increasing computable function $r^e$ such that $\sum d^e_s(W_{r^e(s)})$ is finite.

\medskip

First, we explain why meeting the requirements is sufficient. Let $d$ be a cost function (with the limit condition) obeyed by $A$.
Let $M$ be a positive rational bound on $d$, and let $\hat d = d/M$. Since summation is linear, $A$ obeys $\hat d$.
Let $i$ be such that $\hat d = d^i$. As $f$ is total, the sequence $\seq{A_{f(s)}}$ is a computable approximation of $A$. By Proposition \ref{prop_limit_obedience}, there is an increasing computable function $h$ such that $\sum d^i_s(A_{(f\circ h)(s)})\le 1$. There is an index $e$ such that $\seq{d^e_s} = \seq{d^i_s}$ and $h^e = h$. Then requirement $S^e$ ensures that $W$ obeys $\hat d$. By linearity again, $W$ obeys $d$ as well. 

\

We now discuss how to use the cost function $c$ to help meet a requirement $S^e$. Suppose, for now, that $f$ is the identity function, and that $r^e = h^e$. Let $u = h^e(t)$ be a stage in $\range h^e$. Let $z$ be the least such that $A_{u+1}(z)\ne A_u(z)$. Then we have to enumerate a pair $(z,i)$ into $W_{u+1}$. This, in turn would mean that $\sum d^e_s (W_{r^e(s)})$ will increase by something in the order of $d^e_t(z)$. To keep $\sum d^e_s (W_{r^e(s)})$ bounded, we need to \emph{charge} this cost to some account. There are two possible accounts: the sum $\sum c_s(A_s)$ and the sum $\sum d^e_s (A_{h^e(s)})$. 

Ideally, we define $c_{u+1}(z) \ge d^e_t(z)$. Let $v = h^e(t+1)$ (which we may assume is greater than $u+1$). There are two possibilities: 
\begin{itemize}
	\item If $A_{v}(z) \ne A_{u}(z)$ (for example, if $A(z)$ does not change back between stages $u+1$ and $v$), then a cost of $d^e_{t+1}(z)\ge d^e_t(z)$ is added to the sum $\sum d^e_s(A_{h^e(s)})$. 
	\item If $A_{v}(z) = A_{u}(z)$ then $A_{v}(z)\ne A_{u+1}(z)$ and so a cost of at least $c_{u+1}(z)$ is added to the sum $\sum c_s(A_s)$. Since we defined $c_{u+1}(z)\ge d^e_t(z)$, again a cost at least as large as that facing $\sum d^e_s(W_{r^e(s)})$ is borne by $\sum c_s(A_s)$. 
\end{itemize}

It is important to note that our action at stage $u+1$ to assuage requirement $S^e$ does not require us to wait until we see $v = h^e(t+1)$; it allows us to keep defining $c$ (and $f$) even if $h^e$ is partial. 

\medskip

The catch is that we used the values of $A_u$ and $A_{u+1}$ in order to define $c_{u+1}$. Our commitment to make $\seq{c_s}$ total even if $\seq{A_s}$ is not, means that our definition of $\seq{c_s}$ must be \emph{quicker} than the unfolding of the values of $\seq{A_s}$. For $s<\w$, let $\bar s$ be the greatest number below $s$ such that $A_u(x)$ has converged by stage $s$ for all $u,x\le \bar s$. Usually, $\bar s$ will be much smaller than $s$. At stage $s$ we need to define $c_s$, but can read the values of $\seq{A_u}$ only for $u\le \bar s$. 

This is where the function $f$ comes into play. The speed-up of the approximation of $A$ that it allows us to define can be used to prevent unwanted elements from entering $W$, if $A$ changes back. We return to the situation above, this time with $f$ growing quickly, but still with $r^e = h^e$. Suppose that $n = h^e(t)$ and $u = f(n)$, and $s$ is a stage with $\bar s = u+1$. We see that $A_{u+1}(z)\ne A_u(z)$, and so at stage $s$ we see that we would have liked to define $c_{u+1}(z) \ge d^e_t(z)$. But $s$ is much greater than $u+1$; at stage $u+1$, we were not aware of this situation, and so kept $c_{u+1}(z)$ small. At stage $s$ we would like to rectify the situation by defining $s = f(n+1)$ and $c_s(z)\ge d^e_t(z)$. Let $v = f(h^e(t+1))$, which is presumably greater than $s$. We now have two possibilities:
\begin{itemize}
	\item If $A_s(z) = A_u(z)$, that is, $A(z)$ changed back from its value at stage $u+1$, then the change in $A(z)$ between stages $u$ and $u+1$ need not be recorded in $W$. In this case, $W$ pays no cost related to $z$, and so we do not need to charge anything to anyone. 
	\item Otherwise, the change in $A(z)$ from $u$ to $u+1$ persists at stage $s$, and is recorded in $W$, which pays roughly $d^e_t(z)$. If $A_v(z)= A_s(z)$, then this change persists until stage $v$, and so the cost is paid by the sum $\sum d^e_t(A_{(f\circ h^e)(t)})$. If $A_v(z) = A_u(z)$, then $A(z)$ must have changed at some stage \emph{after stage $s$}, and so the cost can be charged to $\sum c_s(A_s)$. 
\end{itemize}

All is well, except that we did not consider yet another commitment of ours, which is to make $c$ benign (and in fact, to make the bound $g$ uniformly computable from the index $k$ for the partial approximation $\seq{A_s}$). The idea is to again charge increases in $c_s(z)$ to either the sum $\sum c_s(A_s)$ or the sums $\sum d^e_t(A_{(f\circ h^e)}(t))$. That is, in the scenario above, before defining $c_s(z)\ge d^e_t(z)$, we would like to have evidence that $A_s(z)\ne A_u(z)$, so the cost would actually be paid by one of the sums. To avoid this seeming circularity, we ``drip feed'' cost in tiny yet increasing steps. In the scenario above, at stage $s_0 = s$, we would increase $c_{s_0}(z)$ by a little bit -- not all the way up to $d^e_t(z)$ -- and wait for a stage $s_1>s_0$ at which we see what $A_{s_0}(z)$ is (that is, for a stage $s_1$ such that $\bar s_1 \ge s_0$). If $A_{s_0}(z)=A_u(z)$ then we can let $f(n+1)=s_0$. We increased $c_{s_0}(z)$ by something comparable to $c_{u+1}(z)$, and the change in $A(z)$ between stage $u+1$ and stage $s$ shows that this amount was added to $\sum c_s(A_s)$. If $A_{s_0}(z)\ne A_u(z)$, then we increase $c_{s_1}(z)$ again (we can double it), but again not necessarily all the way up to $d^e_t(z)$, and repeat, \emph{while delaying the definition of $f(n+1)$}. Also, since there are infinitely many requirements $S^e$, we have to scale our target, so that only finitely many such requirements affect the $\epsilon$-increases in $\seq{c_s}$; that is, instead of a target of $d^e_t(z)$, we look for $c(z)$ to reach $2^{-(e+1)}d^e_t$. 

\medskip

The last ingredient in the proof is the function $r^e$ -- we have not yet explained why we need $r^e$ to provide an even faster speed-up of $\seq{A_s}$, compared with $\seq{A_{(f\circ h^e)(s)}}$. Now the point is that as slow as the definition of $f$ is, the function $h^e$ shows its values even more slowly. After all, even if $\seq{A_s}$ and $f$ are total, many functions $h^e$ are not. In the scenario above, there may be several stages added to the range of $f$ before we see that $h^e(t)=n$. This means that in trying to define $f(n+1)$, as above, we may suddenly see more requirements $S^e$ worry about more inputs $z$, as more stages enter the range of $f\circ h^e$. The argument regarding the scenario above breaks down if the stage $v = f(h^e(t+1))$ is not greater than the stage $s$. 

We use the function $r^e$ to mitigate this problem. To keep our accounting straight, we need to make the range of $r^e$ contained in the range of $h^e$ (otherwise we might introduce more changes which we will not be able to charge to the sum $\sum d^e_t(A_{(f\circ h^e)(t)})$). In our scenario above, we now assume that $n = r^e(t)$ is in the range of $r^e$. The key now is that by delaying the definition of $r^e(t)$, we may assume that $A(z)$ does not change between stage $u = f(n)$ and the last stage currently in the range of $f$; we use here the assumption that $\seq{A_s}$ indeed converges to $A$. And so the strategy above can work, because even though we declared new values of $f$ beyond $u$, at the time we declare that $n\in \range r^e$, we see that these new values would not spoil the application of our basic strategy. 

\

\subsection{Construction}

Let $\seq{A_s}$ be a uniformly computable sequence of partial functions, and let $k$ be a constant. As mentioned above, for all $s<\w$, we let $\bar s$ be the greatest number below $s$ such that for all $x$ and $u$ bounded by $\bar s$, $A_u(x)$ converges at stage $s$. 

We define a uniformly computable sequence $\seq{c_s}$. We start with $c_0(z) = 2^{-z}$ for all $z<\w$. At every stage $s$, we measure our approximation for $\sum c_s(A_s)$; this, of course, would be the sum of the costs $c_u(x_u)$, where $u,x_u\le \bar s$ and $x_u$ is the least $x\le \bar s$ such that $A_u(x_u)\ne A_{u-1}(x_u)$. If at stage $s$ our current approximation for this sum exceeds $2^k$, we halt the construction, and let $c = c_s$. 

Otherwise, we let $\hat s$ be the greatest stage before stage $s$ such that $c_{\hat s}\ne c_{\hat s-1}$. So $c_s = c_{\hat s}$. Stage $\hat{s}-1$ is the last stage before $s$ at which we took some action toward assuaging the fears of various requirements $S^e$, which is a step toward defining a new value of $f$. By the beginning of stage $s>0$, the function $f$ is defined (and increasing) on inputs $0,1,\dots, m_s$; we start with $f(0)=0$. We will ensure that $f(m_s)\le \bar s$. 

For all $e<\w$, we define a function $r^e$. To begin with, $r^e$ is defined nowhere. Once we see that $h^e(0)\converge$, say at stage $s^e$, we define $r^e(0) = h^e(0)$. Henceforth, at the beginning of stage $s> s^e$, the function $r^e$ is defined on $0,1,\dots, t^e_s$, is increasing, and the range of $r^e$ is contained in both the range of $h^e$ and the domain of $f$. That is, $r^e(t^e_s) \le m_s$. 

Similarly to measuring the sum $\sum c_s(A_s)$, for each $e$, we measure the sum $\sum d^e_t(A_{(f\circ h^e)(t)})$; at stage $s$ we add the costs $d^e_t (y_t)$, where $t<s$ is such that $h^e(t)\le m_s$, and $y_t$ is the least such that $A_{f(h^e(t))}(y) \ne A_{f(h^e(t-1))}(y)$. The requirement $S^e$ is only \emph{active} at stage $s$ if this sum, as calculated at this stage, is bounded by 1. 

Let $s>0$. If 
$\bar s \le \hat s$, then we let $c_{s+1} = c_{s}$, and do not change $f$ (so $m_{s+1} = m_s$). Suppose otherwise. Let $e< m_s$ and $z<m_s$. We say that the requirement $S^e$ is \emph{worried about $z$} at stage $s$ if $S^e$ is active at stage $s$, $t^e_s$ is defined (that is, $s> s^e$), and:
\begin{itemize}
	\item $A_{\bar s}(z) \ne A_{f(r^e(t^e_s))}(z)$; and
	\item $c_{s}(z) < 2^{-(e+1)} d^e_{t^e_s}(z)$. 
\end{itemize} 
Now there are two cases:
\begin{enumerate}
	\item If for all $e<m_s$ and $z<m_s$, the requirement $S^e$ is not worried about $z$ at stage $s$, then we add $m_s+1$ to $\dom f$ by letting $f(m_s+1) = \bar s$ (so $m_{s+1} = m_s+1$). Note that indeed $f(m_{s+1})\le \bar{s+1}$. We also let $c_{s+1}  = c_{s}$. 
	\item Otherwise, we let $z$ be the least number about which some requirement $S^e$ (with $e<m_s$) is worried at stage $s$. For all $y<m_s$ we let $c_{s+1}(y) = \max\{ c_{s}(y), 2c_{s}(z) \}$. For $y\ge m_s$ we let $c_{s+1}(y) = c_{s}(y)$. We do not change $f$, so $m_{s+1}=m_s$. 
\end{enumerate}

This determines $c_s$, $f$, and $m_{s+1}$ at the end of stage $s$. If $m_{s+1}= m_s$ then this is the end of the stage. Otherwise, we now possibly make changes to the functions $r^e$. Let $e<s$ such that $s^e$ has been observed, that is, such that $t^e_s$ is already defined. If there is some $n \in  (r^e(t^e_s), m_{s+1}]$ such that $n$ is observed to be in the range of $h^e$ at stage $s$, and such that $A_{f(m)}\rest{(t^e_s+1)}$ is constant for $m\in [n,m_{s+1}]$, then we extend $\dom r^e$ by letting $r^e(t^e_s+1)$ be the least such $n$; so $t^e_{s+1} = t^e_s+1$. Note that we can enquire about the values of $A_{f(m)}\rest {(t^e_s+1)}$ because $f(m_{s+1}) = \bar s$ and $t^e_s \le m_s < \bar s$. 

If there is no such $n$, then we leave $r^e$ unchanged (so $t^{e}_{s+1} = t^e_s$). This concludes the construction. 

\subsection{Verification}

%
%

The sequence $\seq{c_s}$ is total. Each function $c_s$ is non-increasing, and its limit is 0. For all $z$, $c_{s+1}(z)\ge c_{s}(z)$. 

\begin{lemma}\label{lem_c_bounded_by_1}
	For all $s$ and $z$, $c_s(z)\le 1$. 
\end{lemma}

\begin{proof}
	By induction on $s$. The point is that if at stage $s$ we let $c_{s+1}(y) = 2c_s(z)$ for $z$ as described in the construction, then for some $e<m_s$ we have	
	\[ c_{s}(z) < 2^{-e-1} d^e_{t^e_s}(z) \le 1/2,\]
	because by assumption $d^e_t(z)\le 1$ for all $t$ and $z$. So
	\[ c_{s+1}(y)= 2c_{s}(z) \le 1. \qedhere\]
\end{proof}

Let $c = \lim c_s$. We show that $\seq{c_s}$ witnesses that $c$ is benign, and so $c$ satisfies the limit condition. Moreover, the benignity bound for $\seq{c_s}$ is uniformly computable from $k$. Fix $\epsilon > 0$. Let $m_1(\epsilon), \dots, m_{k(\epsilon)}(\epsilon)$ be the markers associated with $\seq{c_s}$. To avoid confusion with the parameters $m_s$ of the construction, and for notational convenience, for $i\le k(\epsilon)$ let $s_i = m_i(\epsilon)-1$. So for all $i\in \{1,\dots, k(\epsilon)\}$, we have $c_{s_i+1}(s_{i-1}+1)\ge \epsilon$ but $c_{s_i}(s_{i-1}+1)< \epsilon$. 

Fix $N$ such that $2^{-N} < \epsilon/2$. For $e<\w$ let $\delta_e = 2^{-e-1}$.  

\begin{lemma}\label{lem}
	$k(\epsilon) \le 2+2^{k+N} + N^2 (1 + 2^N + 2^{k+N})$. 
\end{lemma}

\begin{proof}
 As mentioned above, for all positive $i\le k(\epsilon)$,  $c_{s_i+1}(s_{i-1}+1)\ge \epsilon$ but $c_{s_i}(s_{i-1}+1)< \epsilon$. In particular, $c_{s_{i}+1}\ne c_{s_{i}}$. This implies that $m_{s_i+1} = m_{s_i}$.
	
	Let $i\in \{2,\dots, k(\epsilon)-1\}$. First, we note that $m_{s_{i+1}} > m_{s_{i}}$. For suppose that $m_{s_{i+1}}= m_{s_{i}}$; so $m_{s_{i+1}+1}= m_{s_i}$. Since $m_{s_{i}}\le s_{i}$, and at any stage $s$ we only increase $c(y)$ for $y<m_s$, we see that $c_{s_{i+1}+1}(s_{i}+1) = c_{s_{i+1}}(s_i+1)$, for a contradiction.
	
	Hence, there is a stage $u_i \in (s_i,s_{i+1})$ at which $f(m_{s_i})$ is defined; at stage $u_i$, no requirement $S^e$ for $e< m_{s_i}$ is worried about any number below $m_{s_i}$. Let $z_i$ be the least number $z$ below $m_{s_i}$ about which some requirement $S^e$ (for $e< m_{s_i}$) worries at stage $s_i$; let $e_i<m_{s_i}$ be such that $S^{e_i}$ worries about $z_i$ at stage $s_i$. 
		
By definition of ``worrying'', 
	\[ c_{s_i}(z_i) < \delta_{e_i}
	d^{e_i}_{t_i}(z_i) \le  \delta_{e_i}
\] where $t_i = t^{e_i}_{s_i}$. On the other hand, since $c_{s_i+1}(s_{i-1}+1) > c_{s_i}(s_{i-1}+1)$, we have 
	 \begin{equation}
		 \epsilon \le c_{s_i+1}(s_{i-1}+1) = c_{s_i+1}(z_i) = 2c_{s_i}(z_i). \label{eqn_epsilon}
	 \end{equation}
	 Hence 
	 $\epsilon \le 2\delta_{e_i} = 2^{-e_i}$, whence $e_i < N$. 
	
	\medskip
	
	Since no new values of $f$ are defined at any stage $s\in [s_i,u_i)$, no new values of $r^{e_i}$ are defined at such stage either; so $t^{e_i}_{u_i} = t_i$. 	
	
	The requirement $S^{e_i}$ does not worry about any number at stage $u_i$, in particular, not about $z_i$. Why not? There are two possibilities:
	\begin{enumerate}
		\item either 
		$A_{\bar {u_i}}(z_i) = A_{f(r^{e_i}(t_i))}(z_i)$, whereas we had $A_{\bar {s_i}}(z_i) \ne A_{f(r^{e_i}(t_i))}(z_i)$; or
		\item case (1) fails, and $c_{u_i}(z_i)\ge \delta_{e_i}
		d^{e_i}_{t_i}(z_i)$, whereas we had $c_{s_i}(z_i)< \delta_{e_i}
		d^{e_i}_{t_i}(z_i)$. 
	\end{enumerate}

Let $I$ be the set of stages $s_i$ (where $i\in \{2,\dots, k(\epsilon)-1\}$) for which case (1) holds. For $e<N$, let $J(e)$ be the set of stages $s_i$ for which case (2) holds and $e_i = e$. Since $e_i<N$ for all $i$, we have $k(\epsilon) \le 2 + |I| + \sum_{e<N} |J(e)|$. Then the lemma is established by showing that $|I|\le 2^{k+N}$ and for all $e<N$, $|J_e|\le N (1 + 2^N + 2^{k+N})$. 

\medskip

In the first case, 	$A_{\bar{s_i}}(z_i) \ne A_{\bar{u_i}}(z_i)$. So there is some $v\in (\bar{s_i},\bar{u_i}]$ such that $A_v(z_i)\ne A_{v-1}(z_i)$; an amount of at least $c_v(z_i)$ is thus added to the sum $\sum c_s(A_s)$, \emph{as is measured at stage $s_{k(\epsilon)}$}. 

Now because $c_{s_i+1}\ne c_{s_i}$, we have $\bar{s_i} > \hat{s_i}$, whence (using Equation \eqref{eqn_epsilon})
\[ c_v(z_i) \ge c_{\bar{s_i}}(z_i) \ge c_{\hat s_i}(z_i) = c_{s_i}(z_i) \ge \epsilon/2.\]
So $c_v(z_i)\ge 2^{-N}$. Stage $s_i$ \emph{charges} the increase of $c_{s_i+1}(s_{i-1}+1)$ beyond $\epsilon$ against $\sum c_s(A_s)$, as measured at stage $s_{k(\epsilon)}$. As the construction is still active at stage $s_{k(\epsilon)}$, this sum is bounded by $2^k$. And the charges for distinct stages $s_i$ are disjoint: we have $v\in (\bar {s_i},\bar {s_{i+1}}]$, because $u_i < s_{i+1}$. Hence  $|I|\le 2^k / 2^{-N} = 2^{k+N}$. 

\medskip

Now consider the second case. Let $e<N$. For $t<\w$, let 
\[ J(e,t) = \left\{ s_i\in J(e)\,:\, t_i = t \right\}.\] 
We first fix $t$ and find a bound on the size of $J(e,t)$. Let $s_i<s_j$ be two elements of $J(e,t)$ (assuming that $|J(e,t)|\ge 2$). Let $z\ge z_i$. If $z\ge m_{s_i}$, then as $t = t_i = t^e_{s_i}< m_{s_i}$, we have $z>t$, and so $d^e_t(z) = 0$. If $z\in [z_i,m_{s_i})$, let $v$ be the last stage before stage $u_i$ such that $c_{v+1}(z_i)\ne c_v(z_i)$; we have $v\ge s_i$. So 
\[ c_{v+1}(z_i) = c_{u_i}(z_i) \ge \delta_e
d^e_t(z_i).\]  Since $z\ge z_i$ and $z<m_{s_i}= m_v$, we actually have $c_{v+1}(z) = c_{v+1}(z_i)$ (they both equal $2c_v(y)$ for some $y\le z_i$, the point being that $c_v(z)\le c_v(z_i)< 2c_v(y)$). Hence
\[ c_{s_j}(z) \ge c_{v+1}(z) = c_{v+1}(z_i) = c_{u_i}(z_i) \ge \delta_e
d^e_t(z_i) \ge \delta_e
d^e_t(z) \]
as $d^e_t$ is non-increasing. So the requirement $S^e$ does not worry about $z$ at stage $s_j$; so $z_j<z_i$. 
In turn, this implies that 
\[ c_{s_j+1}(z_j)= 2c_{s_j}(z_j) \ge 2c_{s_i+1}(z_i) \ge \epsilon ,\]
again using Equation \eqref{eqn_epsilon}. So the quantity $c_{s_j+1}(z_j)$, as $s_j$ varies over the elements of $J(e,t)$, begins with something at least as large as $\epsilon$, and at least doubles with each successive element of $J(e,t)$. Since each $c_s$ is bounded by 1 (Lemma \ref{lem_c_bounded_by_1}), we see that $|J(e,t)|\le N$. 

\medskip

Next, we bound the number of $t$ such that $J(e,t)$ is nonempty. Suppose that $J(e,t)$ is nonempty; let $s_i\in J(e,t)$. Suppose that there is some $t'>t$ such that $J(e,t')$ is nonempty. So $t+1$ is added to $\dom r^e$ at some stage $v$. As we noted above, $t^e_{u_i} = t$, so $v\ge u_i$. At stage $v$ we define $f(m_v+1) = \bar v$. Because $\hat{u_i} \ge s_i+1$ and $\bar{u_i} > \hat{u_i}$, we have $\bar {u_i}> s_i$. So $f(m_v+1) = \bar v > s_i$. Now by the definition of $r^e(t+1)$, as $z_i < t_i$ (because $d^e_{t}(z_i)>0$), we have $A_{f(r^e(t+1))}(z_i) = A_{f(m_v+1)}(z_i) = A_{\bar v}(z_i)$. Because we assume that case (1) fails at stage $s_i$, we have $A_{\bar{u_i}}(z_i)\ne A_{r^e(t)}(z_i)$. This leaves two options: either $A_{f(r^e(t))}(z_i) \ne A_{f(r^e(t+1))}(z_i)$, or $A_{\bar{u_i}}(z_i) \ne A_{\bar v}(z_i)$. 

\begin{itemize}
	\item In the first case, because the range of $r^e$ is contained in the range of $t^e$, we see that an amount of at least $d^e_{t}(z_i)$ is added to the sum $\sum d^e_t(A_{(f\circ h^e)(t)})$, as is calculated at stage $\max J(e)$. Since the requirement $S^e$ is still active at that stage, this sum is bounded by 1. We have $d^e_{t_i}(z_i)\ge \epsilon/(2\delta_{e}) \ge \epsilon$ (Equation \eqref{eqn_epsilon}, using $c_{s_i}(z_i) < \delta_e d^e_{t_i}(z_i)$). So the number of such $t$ is bounded by $2^N$. 
	
	\item In the second case, there is a stage $w\in (\bar{u_i},\bar v]$ such that $A_w(z_i)\ne A_{w-1}(z_i)$. Because $\bar{u_i} > s_i$, we have $c_v(z_i)\ge c_{s_i+1}(z_i)\ge \epsilon$, so an amount of at least $\epsilon$ is added to the sum $\sum c_s(A_s)$ as calculated at stage $s_{k(\epsilon)}$, which as described above, is bounded by $2^k$. The contribution of stages $s_i$ in distinct $J(e,t)$ is counted disjointly, because if $s_j\in J(e,t')$ for $t'>t$ then as $t' = t^e_{s_j}$ we have $s_j > v$, so $w\in (s_i,s_j)$. So the number of such $t$ is bounded by $2^{k+N}$. 
\end{itemize}
Overall, the number of $t$ such that $J(e,t)$ is nonempty is bounded by $1 + 2^N + 2^{k+N}$ as required. 
\end{proof}

\

We now assume that the sequence $\seq{A_s}$ is total and converges to a limit $A$, and that $\sum c_s(A_s)\le 2^k$. The construction is never halted. 

\begin{lemma}\label{lem_f_total}
	The function $f$ is total. 
\end{lemma}

\begin{proof}
	Suppose, for contradiction, that $f$ is not total; at some stage $s^*$ we define the last value $m^* = m_{s^*}+1$ on which $f$ is defined, and for all $s>s^*$ we have $m_s = m^*$. No function $r^e$ is extended after stage $s^*$, so for all $e<\w$, the value $t^e_s$ for all $s>s^*$ is fixed. 
	
	Because $\seq{A_s}$ is total, the function $s\mapsto \bar s$ is unbounded. So there are infinitely many stages $s>s^*$ for which $\bar s > \hat s$; let $T$ be the collection of these stages. By assumption, at each stage $s\in T$ there is some number $z<m^*$ about which some requirement $S^e$ (for $e<m^*$) worries at that stage. For $e<m^*$ and $z<m^*$, let $T(e,z)$ be the collection of stages $s\in T$ at which $S^e$ worries about $z$. There are some $e<m^*$ and $z<m^*$ such that $T(e,z)$ is infinite. 
	
	Let $t = t^e_s$ for $s\in T$. At each stage $s\in T(e,z)$ we have $c_s(z) < \delta_e d^e_t(z)$. At stage $s$ we define $c_{s+1}(z) = 2c_s(y)$ for some $y\le z$, and $c_s(y)\le c_s(z)$. We note that $c_s(z)>0$ because $c_0(z) = 2^{-z}$. This quickly (i.e.\ in $z$ steps) leads to a contradiction. 
\end{proof}

We let $W$ be the change set of $\seq{A_{f(n)}}$. Then $W$ computes $A$. It remains to show that every requirement $S^e$ is met. Fix $e<\w$. Suppose that $h^e$ is total, and that $\sum d^e_t(A_{(f\circ h^e)(t)})\le 1$. 

\begin{lemma}\label{lem_r_e_total}
	The function $r^e$ is total. 
\end{lemma}

\begin{proof}
	Suppose, for contradiction, that $r^e$ is not total; a final value $t^*$ is added to $\dom r^e$ at some stage $s_0$, so $t^e_s = t^*$ for all $s>s_0$. We note that $S^e$ is active at every stage. 
	
	Let $s_1>s_0$ be a stage such that for all $s\ge s_1$, $A_s\rest {t^*+1} = A\rest{t^*+1}$. Let $k<\w$ such that $f(h^e(k))> s_1$. By Lemma \ref{lem_f_total}, there is some stage $s>s_1$ such that $m_{s+1}> m_s$ and $m_s > h^e(k)$. Then at stage $s$ we are instructed to define $r^e(t^*+1)$, a contradiction. 
\end{proof}

The following lemma concludes the proof of Theorem \ref{thm_cost}. 

\begin{lemma}\label{lem_main_cost}
	The sum $\sum d^e_t(W_{(f\circ r^e)(t)})$ is finite. 
\end{lemma}

\begin{proof}
	Let $s^*$ be the stage at which $t^e_s$ is first defined, that is, the stage at which we define $r^e(0)$. For $t>s^*$, let $y_t$ be the least number such that $W_{f(r^e(t))}(y_t)\ne W_{f(r^e(t-1))}(y_t)$. We need to show that $\sum_{t>s^*} d^e_t(y_t)$ is finite. 
	
	Let $t>s^*$. We may assume that $y_t<t$, for otherwise $d^e_t(y_t)=0$. Let $y_t = (z_t,k)$ for some $k<\w$. So $z_t\le y_t$ (using the standard pairing function), and $A_{f(n)}(z_t)\ne A_{f(n-1)}(z_t)$ for some $n\in (r^e(t-1),r^e(t)]$. Taking the least such $n$, we have $A_{f(n-1)}(z_t) = A_{f(r^e(t-1))}(z_t)$ and so $A_{f(n)}(z_t)\ne A_{f(r^e(t-1))}(z_t)$. 
	
	Now there are two possibilities: either $A_{f(n)}(z_t) = A_{f(r^e(t))}(z_t)$, or not. 
	
	\medskip
	
	In the first case, we have $A_{f(r^e(t))}(z_t)\ne A_{f(r^e(t-1))}(z_t)$. Since both $r^e(t-1)$ and $r^e(t)$ belong to the range of $h^e$, and $r^e(x)\ge h^e(x)$ for all $x$, we see that there is some $x\ge t$ such that $A_{f(h^e(x))}(z_t)\ne A_{f(h^e(x-1))}(z_t)$. This means that an amount of at least $d^e_x(z_t)\ge d^e_t(z_t)$ is added to the sum $\sum d^e_t(A_{(f\circ h^e)(t)})$, at a stage $x$ such that $h^e(x)\in (r^e(t-1), r^e(t)]$. Thus the charges for distinct such stages $t$ are disjoint. This shows that the total contribution to the sum $\sum d^e_t (W_{(f\circ r^e)(t)})$ by stages $t$ falling under the first case is at most 1. 
	
	\medskip
	
	In the second case, let $s$ be the stage at which $n$ is added to $\dom f$, that is, $n= m_{s+1} > m_{s}$. The main point is that $t^e_s = t-1$. For $r^e(t) > n = m_{s}$, so $t^e_s < t$. But if $t^e_s< t-1$, let $u\ge s$ be the stage at which $r^e(t-1)$ is defined. The number $m_{u+1}$ is in the range of $h^e$, and $m_{u+1}\ge m_{s+1} = n$. Since $z_t \le t-1$, the condition for defining $r^e(t-1)$ at stage $u$ implies that $A_{f(m)}(z_t)$ is constant for all $m\in [r^e(t-1), m_{u+1}]$, but $A_{f(r^e(t-1))}(z_t)\ne A_{f(n)}$. 
	
	At stage $s$, $S^e$ is not worried about $z_t$ (note that $z_t\le t-1 = t^e_s < m_s$). The requirement $S^e$ is active at stage $s$. We have $f(n) = \bar s$, and $t^e_s = t-1$, and $A_{f(r^e(t-1))}(z_t) \ne A_{f(n)}(z_t)$, which rewriting gives $A_{f(r^e(t^e_s))}(z_t)\ne A_{\bar s}(z_t)$. So the only reason that $S^e$ does not worry about $z_t$ at stage $s$ is that $c_s(z_t) \ge \delta_e d^e_t(z_t)$. Now $f(n)=\bar s > \hat s$, and $c_s(z_t) = c_{\hat s}(z_t)$; so altogether, we see that $c_{f(n)}(z_t) \ge \delta_e d^e_t(z_t)$. 	
	
	 Because this is the second case, we have $A_{f(n)}(z_t) \ne A_{f(r^e(t))}(z_t)$, so there is some stage $u\in (f(n),f(r^e(t))]$ such that $A_u(z_t)\ne A_{u-1}(z_t)$. As $u> f(n)$ we have $c_u(z_t)\ge c_{f(n)}(z_t) \ge \delta_e d^e_t(z_t)$. So an amount of at least $\delta_e d^e_t(z_t)$ is added to the sum $\sum c_s(A_s)$. We have $u\in (f(r^e(t-1)),f(r^e(t))]$ so the charges for distinct $t$ are disjoint. So the total amount contributed to the sum $\sum d^e_t (W_{(f\circ r^e)(t)})$ by stages $t$ falling under the second case is bounded by $2^k/\delta_e$, which is finite. 	
\end{proof}


\end{document}